\documentclass[reqno,twoside]{amsart}

\usepackage{amsfonts}
\usepackage{amsmath,amssymb,amsthm,amsthm}
\usepackage{mathtools}
\usepackage{etoolbox}

\usepackage{color}
\usepackage[english]{babel}
\usepackage{microtype}
\usepackage{fullpage}
\usepackage{floatrow}

\usepackage{float}

\usepackage{placeins} 
\usepackage{flafter} 

\DisableLigatures{encoding = *, family = * }

\newtheorem{theorem}{Theorem}[section]
\newtheorem{definition}[theorem]{Definition}
\newtheorem{lemma}[theorem]{Lemma}

\newtheorem{remark}[theorem]{Remark}

\numberwithin{equation}{section}

\def\RR{{\mathbb{R}}}
\def\NN{{\mathbb{N}}}
\newcommand{\C}{\mathcal{C}}
\newcommand{\norm}[2]{{\left\|#1\right\|}_{#2}}
\newcommand{\fl}[2]{(-d_x^{\,2})^#1#2}
\newcommand{\M}{\mathcal M(\omega\times(0,T_{\rm min}))}

\title{Controllability of the one-dimensional fractional heat equation under positivity constraints}

\author{Umberto Biccari\textsuperscript{$\,\S$,$\dagger$}}  
\address{\textsuperscript{$\S$}\,Chair of Computational Mathematics, Fundaci\'on Deusto Av. de las Universidades 24, 48007 Bilbao, Basque Country, Spain.}
\address{\textsuperscript{$\dagger$}\,Facultad de Ingenier\'ia, Universidad de Deusto, Avenida de las Universidades 24, 48007 Bilbao, Basque Country, Spain.}
\email{umberto.biccari@deusto.es, u.biccari@gmail.com}

\thanks{This project has received funding from the European Research Council (ERC) under the European Union's Horizon 2020 research and innovation programme (grant agreement NO. 694126-DyCon). The work of the three authors is partially supported by the Air Force Office of Scientific Research under Award NO: FA9550-18-1-0242. The work of the first and of the third author was partially supported by the Grant MTM2017-92996-C2-1-R COSNET of MINECO (Spain) and by the ELKARTEK project KK-2018/00083 ROAD2DC of the Basque Government. The work of the third author was partially supported by the Alexander von Humboldt-Professorship program, the European Union’s Horizon 2020 research and innovation programme under the Marie Sklodowska-Curie grant agreement NO. 765579-ConFlex, by the CNCS-UEFISCDI Grant NO: PN-III-P4-ID-PCE-2016-0035 and by the Grant ICON-ANR-16-ACHN-0014 of the French ANR}

\author{Mahamadi Warma\textsuperscript{$\ast\ast$}}  
\address{\textsuperscript{$\ast\ast$}\,University of Puerto Rico, Rio Piedras Campus, Department of Mathematics,  College of Natural Sciences,  17 University AVE. STE 1701  San Juan PR 00925-2537 (USA).}
\email{mahamadi.warma1@upr.edu, mjwarma@gmail.com}

\author{Enrique Zuazua\textsuperscript{$\ast$,$\S$,$\ddagger$}}
\address{\textsuperscript{$\ast$}\, Chair in Applied Analysis, Alexander von Humboldt-Professorship, Department of Mathematics  Friedrich-Alexander-Universit\"at, Erlangen-N\"urnberg, 91058 Erlangen, Germany.}
\address{\textsuperscript{$\ddagger$}\, Departamento de Matem\'aticas, Universidad Aut\'onoma de Madrid, 28049 Madrid, Spain.} 
\email{enrique.zuazua@fau.de}


\begin{document}

\begin{abstract}
In this paper, we analyze the controllability properties under positivity constraints on the control or the state of a one-dimensional heat equation involving the fractional Laplacian $\fl{s}{}$ ($0<s<1$) on the interval $(-1,1)$. We prove the existence of a minimal (strictly positive) time $T_{\rm min}$ such that the fractional heat dynamics can be controlled from any initial datum in  $L^2(-1,1)$ to a positive trajectory through the action of a positive control, when $s>1/2$. Moreover, we show that in this minimal time constrained controllability is achieved by means of a control that belongs to a certain space of Radon measures. We also give some numerical simulations that confirm our theoretical results.
\end{abstract}

\maketitle

\section{Introduction}\label{intro_sec}

The main purpose of the present paper is to completely analysis the constrained controllability properties of the  heat-like equation involving the fractional Laplacian on $(-1,1)$. That is, the system 
\begin{equation}\label{frac_heat_intro}
	\begin{cases}
		z_t+\fl{s}{z} = u\chi_{\omega\times (0,T)}, \quad &\mbox{ in }\; (-1,1)\times(0,T),\quad  s\in(0,1),
		\\
		z=0, &\mbox{ in } (\mathbb R \setminus (-1,1))\times (0,T),
		\\
		z(\cdot,0)=z_0,\; &\mbox{ in }\; (-1,1),
	\end{cases}
\end{equation}
where the fractional Laplace operator $\fl{s}$ ($0<s<1$) is defined for a smooth function $v$ by the following singular integral:
\begin{align*}
	\fl{s}{v}(x):= c_s\,\mbox{P.V.}\int_{\RR}\frac{v(x)-v(y)}{|x-y|^{1+2s}}\,dy,\;\;x\in \RR.
\end{align*}
We refer to Section \ref{sec-main-results} for more details. 

In \eqref{frac_heat_intro}, the solution $z$ is the state to be controlled and $u$ is our control function which is localized in an open set $\omega\subset (-1,1)$.

The controllability properties of the fractional heat equation in open subsets of $\RR^N$ ($N\ge 2$) are still not fully understood by the mathematical community. The classical tools (see e.g. \cite{Zua1} and the references therein) like the Carleman estimates usually used to study the controllability for heat equations are still not available for the fractional Laplacian (except on the whole space $\RR^N$). For this reason, our analysis in the present article is limited to the one-dimensional case. Another difficulty for analyzing the system \eqref{frac_heat_intro} by using some spectral properties is that contrarily to the local case $s=1$ where the eigenvalues and eigenfunctions of the system are well known, for the fractional case, we just know an asymptotic for the eigenvalues and an explicit  formula for the eigenfunctions is not accessible.

In the absence of constraints, the fractional heat equation \eqref{frac_heat_intro} is null-controllable in any positive time $T>0$, provided that $s>1/2$. This has been proved in \cite{biccari2017controllability} by using the gap condition on the eigenvalues, and has been validated through numerical experiments. In space dimension $N\ge 2$, the best possible controllability result available for the fractional heat equation is the approximate controllability recently obtained in \cite{KeWa} for interior controls and in \cite{WarC} for exterior controls. 

In this work, we have obtained the following specific results: 
\begin{itemize}
	\item[(i)] Firstly, we shall show in Theorem \ref{control_thm_unconstr} that if $s>1/2$, then the system \eqref{frac_heat_intro} is controllable from any given initial datum in $L^2(-1,1)$ to zero (and, by translation, to trajectories) in any positive time $T>0$ by means of $L^\infty$-controls. This extends the analysis of \cite{biccari2017controllability}, where only the classical case of $L^2$-controls was considered. The proof will use the canonical approach of reducing the question of controllability with an $L^\infty$-control to a dual observability problem in $L^1$, and the use of Fourier series expansions to obtain a new result on the $L^1$-observation of linear combinations of real exponentials.

	\item[(ii)] Secondly, as a consequence of our first result, in Theorem \ref{control_thm} we shall prove the existence of a minimal (strictly positive) time $T_{\rm min}$ such that the fractional heat dynamics \eqref{frac_heat_intro} can be controlled to positive trajectories through the action of a positive control. Moreover, if the initial datum is supposed to be positive as well, then the maximum principle guarantees the positivity of the states too. 
\end{itemize}

Fractional order operators (in particular the fractional Laplace operator) have recently emerged as a modelling alternative in various branches of science.  They usually describe anomalous diffusion. A number of stochastic models for explaining anomalous diffusion have been introduced in the literature. Among them we quote the fractional Brownian motion, the continuous time random walk, the L\'evy flights, the Schneider gray Brownian motion, and more generally, random walk models based on evolution equations of single and distributed fractional order in  space (see e.g. \cite{DS,GR,Man,Sch}). In general, a fractional diffusion operator corresponds to a diverging jump length variance in the random walk. 

In many PDEs models some constraints need to be imposed when considering practical applications. This is for instance the case of diffusion processes (heat conduction, population dynamics, etc.) where realistic models have to take into account that the state represents some physical quantity which must necessarily remain positive (see e.g. \cite{chan1990optimal}). 

This topic is also related to some other relevant applications, like the optimal management of compressors in gas transportation networks requiring the preservation of severe safety constraints (see e.g. \cite{colombo2009optimal,martin2006mixed,steinbach2007pde}). 

Finally, this issue is also important in other PDEs problems based on scalar conservation laws, including (but not limited to) the Lighthill-Whitham and Richards traffic flow models (\cite{colombo2004minimising,lighthill1955kinematic,richards1956shock}) or the isentropic compressible Euler equation (\cite{glass2007controllability}).

Most of the existing controllability theory for PDEs has been developed in the absence of constraints on the controls and/or the state. To the best of our knowledge, the literature on constrained controllability is currently very limited and the majority of the available results do not guarantee that controlled trajectories fulfill the physical restrictions of the processes under consideration. 

In the context of the heat equation, the problem of constrained controllability has been firstly addressed in \cite{loheac2017minimal} in the linear case and has been later extended to semi-linear models in \cite{pighin2018controllability}. In particular, in the mentioned references, the authors proved that the linear and semi-linear heat equations are controllable to any positive steady state or trajectory by means of non-negative boundary controls, provided the control time is long enough. Moreover, for positive initial data, the maximum principle guarantees also that the positivity of the state is preserved. On the other hand, it was also proved that controllability by non-negative controls fails if the time is too short, whenever the initial datum differs from the final target.

In addition to the results for heat-like equations, constrained controllability has been also analyzed in the context of population dynamics. To be more precise, in \cite{hegoburu2018controllability,maity2018controllability} it has been shown that the controllability of Lotka-McKendrick type systems with age structuring can be obtained by preserving the positivity of the state, once again in a long enough time horizon. These results have been recently extended in \cite{maity2018control} to general infinite-dimensional systems with age structure.

The study of the controllability properties under positivity constraints is a very reasonable question for scalar-valued parabolic equations, which are canonical examples where the positivity is preserved for the free dynamics. Therefore, the issue of whether the system can be controlled in between two states by means of positive controls, by possibly preserving also the positivity of the controlled solution, arises naturally. 

The existence of a minimal time for constrained controllability is in counter-trend with respect to the unconstrained case, in which linear and semi-linear parabolic systems are known to be controllable at any positive time. Notwithstanding, it is not surprising in the context of constrained controllability. Indeed, often times, norm-optimal controls allowing to reach the target at the final time are restrictions of solutions of the adjoint system. Accordingly these controls experience large oscillations in the proximity of the final time, which are enhanced when the time horizon of the control is small. This eventually leads to control trajectories that go beyond the physical thresholds and fail to fulfill the positivity constraint (see \cite{glowinski2008exact}). On the other hand, when the time interval is long, we expect the control property to be achieved with controls of small amplitude, thus ensuring small deformations of the state and, in particular, preserving its positivity. 

In order to prevent this highly oscillatory behavior, in some more recent papers (see \cite{fernandez2014numerical} and the references therein) several authors proposed new techniques for the computation of the controls which do not apply duality arguments but explore instead a direct approach in the framework of global Carleman estimates. More precisely, they consider controls that minimize a functional involving weighted integrals of the state and the control itself. Nevertheless, while this methodology does not generate any spurious oscillations in the computed control, it also requires a large enough time horizon for the controllability of the equation. This is a further confirmation that a minimum strictly positive controllability time may appear naturally when imposing constraints on the control. Roughly speaking, by imposing constraints to the control, we are somehow providing an impediment for the state to reach the target, when the control time horizon is long enough. This behavior is then a warning that existing unconstrained controllability results, that are valid within arbitrarily short time, may be unsuitable in practical applications in which state-constraints need to be preserved along controlled trajectories. 

For completeness, we mention that constrained controllability problems analogous to \cite{loheac2017minimal,pighin2018controllability} have been recently treated in \cite{balc2018global}. There, the author considered again the controllability problem for a semi-linear heat equation, this time with interior distributed controls, and proved the exact controllability to positive trajectories. However, we have to stress that, in that reference, no positivity assumptions on the control is considered. Consequently, the controllability result is achieved in any positive time $T>0$.

In addition to the results for parabolic equations, similar questions for the linear wave equation have been analyzed in \cite{pighin2019controllability}. There, the authors proved the controllability to steady states and trajectories through the action of a positive control, acting either in the interior or on the boundary of the considered domain. Nevertheless, in that case control and state positivity are not interlinked. Indeed, because of the lack of a maximum principle, the sign of the control does not determine the sign of the solution whose positivity is no longer guaranteed. 

The rest of the paper is organized as follows. We state the main result of the paper in Section \ref{sec-main-results}.
In Section \ref{preliminary_sec}, we start by presenting some preliminary technical results that are needed throughout the paper. Moreover, we give there the proof of the unconstrained controllability of the fractional heat equation \eqref{frac_heat_intro} with $L^\infty$-controls. Section \ref{pr-main-rel} is devoted to the proof of the main result, namely, Theorem \ref{control_thm} below. The proof is divided in three parts. In Section \ref{control_thm_sec}, we prove the first part concerning the constrained controllability of \eqref{frac_heat}. In Section \ref{pos_time_thm_sec}, we obtain the strict positivity of the minimal controllability time $T_{\rm min}$. Section \ref{measure_control_thm_sec} is devoted to the proof of the controllability in minimal time by means of measure controls. In Section \ref{numerical_sec}, we present some numerical simulations validating our theoretical results. Finally, in Section \ref{sec-con-rem} we give some concluding remarks and propose some open problems.

\section{Problem formulation and main results}\label{sec-main-results}

In this section, we formulate precisely the problem we would like to investigate and we state our main results.

Let $T>0$ be a real number, and define $Q:=(-1,1)\times (0,T)$ and $Q^c:=(-1,1)^c\times (0,T)$ where $(-1,1)^c:=\RR\setminus(-1,1)$. Consider the following controllability problem for the fractional heat equation:
\begin{align}\label{frac_heat}
	\begin{cases}
		z_t+\fl{s}{z} = u\chi_{\omega\times(0,T)} & \mbox{ in } \,Q,
		\\
		z = 0 & \mbox{ in } \,Q^c,
		\\
		z(\cdot,0) = z_0(\cdot) & \mbox{ in } \,(-1,1).
	\end{cases}
\end{align}

In \eqref{frac_heat}, $\omega\subset (-1,1)$ is the control region, $u$ is the control function and $z$ is the state to be controlled, while for $s\in(0,1)$, the operator $\fl{s}{}$ is the fractional Laplacian, defined for any function $v$ sufficiently smooth as the following singular integral:
\begin{align}\label{fl}
	\fl{s}{v}(x):= c_s\,\mbox{P.V.}\int_{\RR}\frac{v(x)-v(y)}{|x-y|^{1+2s}}\,dy,\;\;x\in \RR,
\end{align}
with $c_s$ an explicit normalization constant (see e.g. \cite{di2012hitchhiker}).  In Section \ref{preliminary_sec} we shall specify the difference between the considered problem \eqref{frac_heat} and the case of the spectral fractional Laplace operator.

It is known (see \cite{biccari2017controllability}) that the fractional heat equation \eqref{frac_heat} is null controllable in any time $T>0$ by means of a control $u\in L^2(\omega\times (0,T))$ if and only if $s>1/2$. In other words, given any $z_0\in L^2(-1,1)$ and $T>0$, there exists a control function $u\in L^2(\omega\times (0,T))$ such that the corresponding unique solution $z$ of \eqref{frac_heat} satisfies $z(x,T)=0$ a.e. in $(-1,1)$. If $s\leq 1/2$, instead, null-controllability cannot be achieved and the system turns out to be only approximately controllable. Besides, the equation being linear, by translation the same result holds if the final target is a trajectory $\widehat{z}$.

Moreover, it is also known (see Lemma \ref{lem} below) that the fractional heat equation preserves positivity. More precisely, if $z_0$ is a given non-negative initial datum in $L^2(-1,1)$ and $u$ is a non-negative function, then so it is for the solution $z$ of \eqref{frac_heat}. Hence, the following question arises naturally: 
\begin{quotation}
	{\bf Can we control the fractional heat dynamics \eqref{frac_heat} from any initial datum $z_0\in L^2(-1,1)$ to any positive trajectory $\widehat{z}$, under positivity constraints on the control and/or the state?}
\end{quotation}

In other words we want to analyze whether it is possible to choose a control $u\geq 0$ steering the solution of \eqref{frac_heat} from $z_0\in L^2(-1,1)$ to a positive trajectory $z(\cdot,T)=\widehat{z}(\cdot,T)>0$, while possibly maintaining this solution non-negative along the whole time interval, i.e., 
\begin{align*}
	z(x,t)\geq 0 \;\; \textrm{ for every }\;\; (x,t)\in (-1,1)\times (0,T).
\end{align*}

Clearly, we are only interested in the case $\widehat{z}(\cdot,T)\neq z_0$. Otherwise, the trajectory $z\equiv z_0=\widehat{z}(\cdot,T)$ trivially solves the problem.

As we will see, the answer to the above question is positive, provided that the controllability time is large enough. In particular, our main results in the present paper are the following.

\begin{theorem}\label{control_thm}
Let $s>1/2$, $z_0\in L^2(-1,1)$ and let $\widehat{z}$ be a positive trajectory, i.e., a solution of \eqref{frac_heat} with initial datum $0< \widehat{z}_0\in L^2(-1,1)$ and right hand side $\widehat{u}\in L^\infty(\omega\times(0,T))$. Assume that there exists $\nu> 0$ such that $\widehat{u}\geq\nu$ a.e in $\omega\times(0,T)$. Then, the following assertions hold.
\begin{itemize}
	\item[(I)] There exist $T>0$ and a non-negative control $u\in L^\infty(\omega\times(0,T))$ such that the corresponding solution $z$ of \eqref{frac_heat} satisfies $z(x,T) = \widehat{z}(x,T)$ a.e. in $(-1,1)$. Moreover, if $z_0\geq 0$, then $z(x,t)\geq 0$ for every $(x,t)\in (-1,1)\times (0,T)$. 	
		
	\item[(II)] Define the minimal controllability time by
	\begin{align}\label{t_min}
		T_{\rm min}(z_0,\widehat{z}):= \inf\Big\{T>0:\exists\;\; 0&\leq u\in L^\infty(\omega\times(0,T)) \textrm{ s.t. } \notag 
		\\
		& z(\cdot,0)=z_0 \textrm{ and } z(\cdot,T)=\widehat{z}(\cdot,T)\Big\}.
	\end{align}
	Then, $T_{\rm min}>0$.
		
	\item[(III)] For $T = T_{\rm min}$, there exists a non-negative control $u\in\M$, the space of
	Radon measures on $\omega\times(0,T_{\rm min})$, such that the corresponding solution $z$ of \eqref{frac_heat} satisfies $z(x,T) = \widehat{z}(x,T)$ a.e. in $(-1,1)$.
\end{itemize}
\end{theorem}

This result is in the same spirit as the ones obtained in \cite{loheac2017minimal,pighin2018controllability} in the context of the linear and semi-linear local heat equations under the action of a boundary control. Following the methodology presented in the mentioned references, the first ingredient for proving Theorem \ref{control_thm} is to show that, in absence of constraints, \eqref{frac_heat} is controllable by means of an $L^\infty$-control. This will be given by the following.

\begin{theorem}\label{control_thm_unconstr}
For any $z_0\in L^2(-1,1)$, $s>1/2$ and $T>0$, there exists a control function $u\in L^\infty(\omega\times(0,T))$ such that the corresponding unique weak solution $z$ of \eqref{frac_heat} with initial datum $z(x,0)=z_0(x)$ satisfies $z(x,T)=0$ a.e. in $(-1,1)$. Moreover, there is a constant $\C>0$ (depending only on $T$) such that 
\begin{align}\label{control_bound}
	\norm{u}{L^\infty(\omega\times(0,T))}\leq\C\norm{z_0}{L^2(-1,1)}.
\end{align}
\end{theorem}

By means of a classical duality argument (see \cite{fabre1995approximate,fernandez2000null,micu2012time}), Theorem \ref{control_thm_unconstr} is equivalent to the following observability result for the adjoint equation associated to \eqref{frac_heat}.

\begin{theorem}\label{obs_thm}
For any $T>0$ and $p_T\in L^2(-1,1)$, let $p\in L^2((0,T);H_0^s(-1,1))\cap C([0,T];L^2(-1,1))\cap H^1((0,T);H^{-s}(-1,1))$ be the weak solution of the adjoint system 
\begin{align}\label{frac_heat_adj}
	\begin{cases}
		-p_t+\fl{s}{p} = 0 & \mbox{ in }Q,
		\\
		p = 0 & \mbox{ in }Q^c,
		\\
		p(\cdot,T) = p_T(\cdot) & \mbox{ in }(-1,1).
	\end{cases}
\end{align}
Then, for any $s>1/2$, there is a constant $\C=\C(T)>0$ such that 
\begin{align}\label{obs}
	\norm{p(\cdot,0)}{L^2(-1,1)}^2 \leq \C\left(\int_0^T\int_\omega |p(x,t)|\,dxdt\right)^2. 
\end{align}
\end{theorem}

We refer to Section \ref{preliminary_sec} for the definition of the spaces $H_0^s(-1,1)$ and $H^{-s}(-1,1)$.

We shall prove Theorem \ref{obs_thm} by employing spectral techniques and with the help of the following $L^1$-observability result for linear combinations of real exponentials.

\begin{theorem}\label{L1_obs_thm}
Let $\{\mu_k\}_{k\geq 1}$ be a sequence of real numbers satisfying the following conditions:
\begin{subequations}
	\begin{align}
		&\textrm{1. There exists } \gamma>0 \textrm{ such that } \mu_{k+1}-\mu_k\geq \gamma \textrm{ for all } k\geq 1.\label{spectral_cond1}
		\\
		&\textrm{2. }\sum_{k\geq 1}\frac{1}{\mu_k}<+\infty. \label{spectral_cond2}
	\end{align}
\end{subequations}

Then, for any $T>0$, there is a constant $\mathcal C(T)>0$ (depending only on $T$) such that, for any sequence $\{c_k\}_{k\geq 1}$ it holds the inequality
\begin{align}\label{obs2}
	\sum_{k\geq 1}|c_k|e^{-\mu_k T} \leq \C(T)\, \norm{\sum_{k\geq 1} c_k e^{-\mu_k t}}{L^1(0,T)}.
\end{align} 

Moreover, the function $\C(T)$ is uniformly bounded away from $T=0$ and blows-up exponentially as $T\downarrow 0^+$.
\end{theorem}

Theorem \ref{L1_obs_thm} will follow from the classical M\"untz Theorem for families of real exponentials and from the results of \cite{schwartz1943etude}. 

We stress that Theorem \ref{L1_obs_thm}, of independent interest on its own, is not specific to the equation \eqref{frac_heat} we are considering but it allows for more general results. Indeed, it yields the immediate knowledge of $L^\infty$-controls for one-dimensional problems simply by knowing the explicit spectrum of the equation. 

We remark that the fact that $\C(T)$ blows-up exponentially as $T\downarrow 0^+$ will imply the same for the observability constant in the estimate \eqref{obs}. This is in agreement with the well-known behavior for heat-like processes in the classical case $s=1$, as it has been shown, for instance, in \cite{fernandez2000null}.

\section{Preliminary results}\label{preliminary_sec}
We present here some preliminary results which are needed for the proof of Theorem \ref{control_thm}. We start by introducing the appropriate function spaces needed to study our problem. For any $s\in (0,1)$ and an arbitrary open set $\Omega\subset\mathbb R$, we denote by

\begin{equation*}
	H^{s}(\Omega):=\left\{v\in L^2(\Omega):\;\int_\Omega\int_\Omega
\frac{|v(x)-v(y)|^{2}}{|x-y|^{1+2s}}dxdy<+\infty \right\}
\end{equation*}
the fractional order Sobolev space endowed with the norm
\begin{equation*}
	\norm{v}{H^s(\Omega)}:=\left( \int_\Omega|v|^2\;dx+\int_\Omega\int_\Omega\frac{|v(x)-v(y)|^2}{|x-y|^{1+2s}}
dxdy\right) ^{\frac 12},
\end{equation*}
and we let 
\begin{align*}
	H_0^s(-1,1):=\Big\{v\in H^s(\RR):\; v=0\;\mbox{ on }\;\RR\setminus(-1,1)\Big\}.
\end{align*}

Moreover, we let $H^{-s}(-1,1):=(H_0^s(-1,1))^\star$ be the dual space of $H_0^s(-1,1)$ with respect to the pivot space $L^2(-1,1)$.  Then, we have the following continuous embeddings: $H_0^s(-1,1)\hookrightarrow L^2(-1,1)\hookrightarrow H^{-s}(-1,1)$. Finally, we denote by $H^{s}_{\rm loc}(-1,1)$ the space defined by
\begin{align*}
	H^{s}_{\rm loc}(-1,1)=\Big\{v\in L_{\rm loc}^2(-1,1):\;v\varphi\in H^{s}(-1,1)\;\mbox{ for all }\;\varphi\in\mathcal D(\Omega)\Big\}.
\end{align*}

If $s\ge 1$, then the above spaces are defined as in \cite{BWZ1} and their references.
For more details on fractional order Sobolev spaces we refer to \cite{di2012hitchhiker,War} and their references. 

Next we introduce our notion of solutions to the system \eqref{frac_heat}.
\begin{definition}
We shall say that a function 
$$z\in C([0,T];L^2(-1,1))\cap L^2((0,T);H_0^s(-1,1))\cap H^1((0,T); H^{-s}(-1,1))$$ 
 is a weak solution of the system \eqref{frac_heat} if $z(0,\cdot)=z_0$ a.e. in $(-1,1)$ and the equality
\begin{align*}
	\langle z_t(t,\cdot),v\rangle_{H^{-s}(-1,1),H_0^s(-1,1)} & +\frac{c_s}{2}\int_{\RR}\int_{\RR}\frac{(z(t,x)-z(t,y))(v(x)-v(y))}{|x-y|^{1+2s}}\;dxdy
	\\
	&=\langle u(t,\cdot),v\rangle_{H^{-s}(-1,1),H_0^s(-1,1)}
\end{align*}
holds for every $v\in H_0^s(-1,1)$ and a.e. $t\in (0,T)$.
\end{definition}

We recall that according to \cite[Theorem 26]{leonori2015basic}, for any $u\in L^2((0,T);H^{-s}(-1,1))$ and $z_0\in L^2(-1,1)$, the system \eqref{frac_heat} admits a unique weak solution $z$. 

Moreover, if $u\in L^2(\omega\times(0,T))$ and $z_0\equiv 0$, then it has been shown in \cite[Theorem 1.5]{biccari2018local}  that 
\begin{align*}
	z\in L^2((0,T);H^{2s}_{\rm loc}(-1,1))\cap L^{\infty}((0,T);H^s_0(-1,1))\cap H^1((0,T);L^2(-1,1)).
\end{align*}

Furthermore, as we have mentioned above, the fractional heat equation preserves positivity, meaning that, if $u$ is non-negative and $z_0$ is also non-negative, then the unique solution $z$ of the system \eqref{frac_heat} is also non-negative. Such a result has been stated in \cite[Theorem 26]{leonori2015basic} but without giving a proof. For the sake of completeness we include the full proof here. We state our result in the case $N=1$ but the same holds for $N\ge 1$ without any modification of the proof.

\begin{lemma}\label{lem}
Let $u\in L^2(\omega\times(0,T))$ and  $z_0\in L^2(-1,1)$ be non-negative. Then, the corresponding unique solution $z$ of the system \eqref{frac_heat} is also non-negative.
\end{lemma}

\begin{proof}
Denote by $\fl{s}_D$ the realization of $\fl{s}$ in $L^2(-1,1)$ with the Dirichlet exterior condition $u=0$ in $(-1,1)^c$. Then, $\fl{s}_D$ is the self-adjoint operator in $L^2(-1,1)$ associated with the bilinear, symmetric and closed form $\mathcal E:H_0^s(-1,1)\times H_0^s(-1,1)\to\mathbb R$  and given by
\begin{align*}
	\mathcal E(\varphi,\psi):=\frac{c_s}{2}\int_{\RR}\int_{\RR}\frac{(\varphi(x)-\varphi(y))(\psi(x)-\psi(y))}{|x-y|^{1+2s}}\;dxdy,\;\;\varphi,\psi\in H_0^s(-1,1).
\end{align*}

We claim that  $\fl{s}_D$ is a resolvent positive operator. Indeed, let $\lambda>0$ be a real number, $f\in L^2(-1,1)$ and set 
\begin{align*}
	\phi:=\Big(\lambda +\fl{s}_D\Big)^{-1}f.
\end{align*}
Then, $\phi$ belongs to $H_0^s(-1,1)$ and is a weak solution of the Dirichlet problem
\begin{equation*}
	\begin{cases}
		\fl{s}\phi+\lambda \phi=f\;\;&\mbox{ in }\; (-1,1),
		\\
		\phi=0&\mbox{ in }\; (-1,1)^c,
	\end{cases}
\end{equation*}
in the sense that
\begin{align}\label{29}
	\mathcal E(\phi,v)+\lambda\int_{-1}^1\phi v\;dx=\int_{-1}^1fv\;dx,\;\;\forall\; v\in H_0^s(-1,1).
\end{align}
It is clear that that there is a constant $\C>0$ such that
\begin{align}\label{28}
	\lambda\int_{-1}^1|v|^2\;dx+\mathcal E(v,v)\ge \C\|v\|_{H_0^s(-1,1)}^2
\end{align}
for all $v\in H_0^{s}(-1,1)$. 

Now, assume that $f\le 0$ a.e. in $(-1,1)$ and let $\phi^+ :=\max\{\phi,0\}$. It follows from  \cite{War} that $\phi^+\in  H_0^s(-1,1)$. Let $\phi^- :=\max\{-\phi,0\}$. Since
\begin{align*}
	\Big(\phi^-(x)-\phi^-(y)\Big)\Big(\phi^+(x)-\phi^+(y)\Big) = &\,\phi^-(x)\phi^+(x)-\phi^-(x)\phi^+(y)
	\\
	& -\phi^-(y)\phi^+(x)+\phi^-(y)\phi^+(y)
	\\
	= &\,-\Big(\phi^-(x)\phi^+(y)+\phi^-(y)\phi^+(x)\Big)\le 0,
\end{align*}
we have that $\mathcal E(\phi^-,\phi^+)\le 0$. Hence,
\begin{align*}
	\mathcal E(\phi,\phi^+)=\mathcal E(\phi^+-\phi^-,\phi^+)=\mathcal E(\phi^+,\phi^+)-\mathcal E(\phi^-,\phi^+)\ge 0.
\end{align*}
Then, using \eqref{29} with $v=\phi^+$ we get that
\begin{align*}
	0\le \lambda\int_{-1}^1\phi\phi^+\;dx+\mathcal E(\phi,\phi^+)=\int_{-1}^1f\phi^+\;dx\le 0.
\end{align*}

By \eqref{28}, the preceding estimate implies that $\phi^+=0$ a.e. in $(-1,1)$, that is, $\phi\le 0$ a.e. in $(-1,1)$. We have shown that the resolvent $(\lambda +\fl{s}_D)^{-1}$ is a positive operator. Now it follows from the corresponding result on abstract Cauchy problems (see e.g. \cite[Theorem 3.11.11]{ABHN}) that,  if $u\ge 0$ and $z_0\ge 0$, then the unique solution $z$ of \eqref{frac_heat} is also non-negative.  We notice that this can be also seen from the representation of the solution $z$. More precisely, we have that the solution $z$ of \eqref{frac_heat} is given for a.e. $x\in (-1,1)$ and a.e. $t>0$ by
\begin{align}\label{U0}
z(x,t)=(T(t)z_0)(x)+\int_0^tT(t-\tau)u(x,\tau)\;d\tau,
\end{align}
where $(T(t))_{t\ge 0}$ is the submarkovian (positivity-preserving and $L^\infty$-contractive) semigroup on $L^2(-1,1)$ generated by $-(-d_x^2)_D^s$.  The proof is finished.
\end{proof}

\begin{remark}
{\em 
We mention the following facts.
\begin{enumerate}
	\item Firstly, we notice that the operator $(-d_x^2)_D^s$ defined above is different from the Dirichlet spectral fractional Laplace operator, that is, the fractional $s$-power of the realization in $L^2(-1,1)$ of the Laplace operator $-d_x^2$ with the zero Dirichlet boundary condition. That is, the operator defined by
	\begin{align*}
		(-d_x^2)_Sv: = \sum_{k\geq 1} \langle v,e_k\rangle_{L^2(-1,1)} \eta_k^se_k,
	\end{align*}
	with $e_k$ and $\eta_k$ being the eigenfunctions and eigenvalues of the Dirichlet Laplacian on $(-1,1)$, respectively. 
 	The eigenvalues and eigenfunctions of $(-d_x^2)_D^s$ and $(-d_x^2)_S$ are different. More precisely, the latter operator has the same eigenfunctions as the Dirichlet Laplacian which are smooth functions (by using elliptic regularity), but for the first one, its eigenfunctions are even not Lipschitz continuous on $[-1,1]$. We refer to \cite{BWZ1,Val} for the comparison of the eigenvalues and for more details on this topic. 
	\item Secondly, the techniques we use in this paper would also apply for the fractional heat equation involving $(-d_x^2)_S$. Actually, the proof of the constrained controllability for that equation would still be based on the $L^1$-Ingham-like estimate \eqref{obs2} and would use the same arguments we shall use in Section \ref{pr-main-rel}.
 \end{enumerate}}
 \end{remark}

We can now prove Theorem \ref{control_thm_unconstr} ensuring that, without imposing any constraint on the control, it is possible to obtain the null-controllability of \eqref{frac_heat} by means of a control function $u\in L^\infty(\omega\times(0,T))$. 

To this end, we shall first give the proof of Theorem \ref{L1_obs_thm}, and then use this result to obtain the observability inequality \eqref{obs}.

\begin{proof}[\bf Proof of Theorem \ref{L1_obs_thm}]
The proof of \eqref{obs2} is a direct consequence of the results presented in \cite{schwartz1943etude}. Let us define the function $F: [0,T]\to\mathbb R$ by
\begin{align}\label{F}
	F(t):=\sum_{k\geq 1} c_k e^{-\mu_k t}.
\end{align}
	
According to \cite[Section 9, page 30, Equation (9.a)]{schwartz1943etude}, under the hypothesis \eqref{spectral_cond1} and \eqref{spectral_cond2} the following estimates hold:
\begin{align}\label{est_9a_schwartz}
	\sum_{k\geq 1} |c_k|e^{-\mu_kT}\leq \C(T)\norm{F}{L^1(0,T)},
\end{align}
with $\C(T)$ a positive constant depending only on $T$ and uniformly bounded for all $T\geq\varepsilon>0$. Then, \eqref{obs2} immediately follows by applying \eqref{est_9a_schwartz} to the function $F$ defined in \eqref{F}. Moreover, by using \cite[Section 9, page 32, Equation (9.h')]{schwartz1943etude} and the lower bound of $\C(T)$ given in \cite[Section 9, page 36]{schwartz1943etude}, we can readily check that the constant $\C(T)$ blows-up exponentially as $T\downarrow 0^+$. The proof is finished.
\end{proof}

We can now employ \eqref{obs2} to prove Theorem \ref{obs_thm}. To this end, we will first need the following technical result.

\begin{lemma}
Consider the eigenvalues problem for the Dirichlet fractional Laplacian in $(-1,1)$:
\begin{align}\label{fl_eigenv}
	\begin{cases}
		\fl{s}{\phi_k} = \lambda_k\phi_k, & \mbox{ in } (-1,1)
		\\
		\phi_k = 0, & \mbox{ in } (-1,1)^c.
	\end{cases}
\end{align}

That is, $\{\phi_k\}_{k\in\NN}$ is the orthonormal basis of eigenfunctions of the operator $\fl{s}_D$ defined in the proof of Lemma \ref{lem} with associated eigenvalues $\{\lambda_k\}_{k\in\NN}$.Then for any open set $\omega\subset (-1,1)$, there is a constant $\beta>0$ (independent of $k$) such that 
\begin{align}\label{L1_bound_eigen}
	\norm{\phi_k}{L^1(\omega)}\geq\beta >0.
\end{align}
\end{lemma}

\begin{proof}
The proof is based on the asymptotic results on the spectrum of the fractional Laplacian contained in  \cite{kwasnicki2010spectral,kwasnicki2012eigenvalues}. Let us introduce the following auxiliary function $q:\mathbb R\to\mathbb R$ defined by 
\begin{align}\label{q_def}
	q(x) := \begin{cases}
		0 & x\in \left(-\infty,-\frac 13\right),
		\\[7pt]
		\displaystyle \frac 92 \left(x+\frac 13\right)^2 & x\in \left(-\frac 13,0\right),
		\\[10pt]
		\displaystyle 1-\frac 92 \left(x-\frac 13\right)^2 & x\in \left(0,\frac 13\right),
		\\[7pt]
		1 & x\in \left(\frac 13,+\infty\right).
	\end{cases}
\end{align}
\begin{figure}[h]
	\centering
	\includegraphics[scale=0.5]{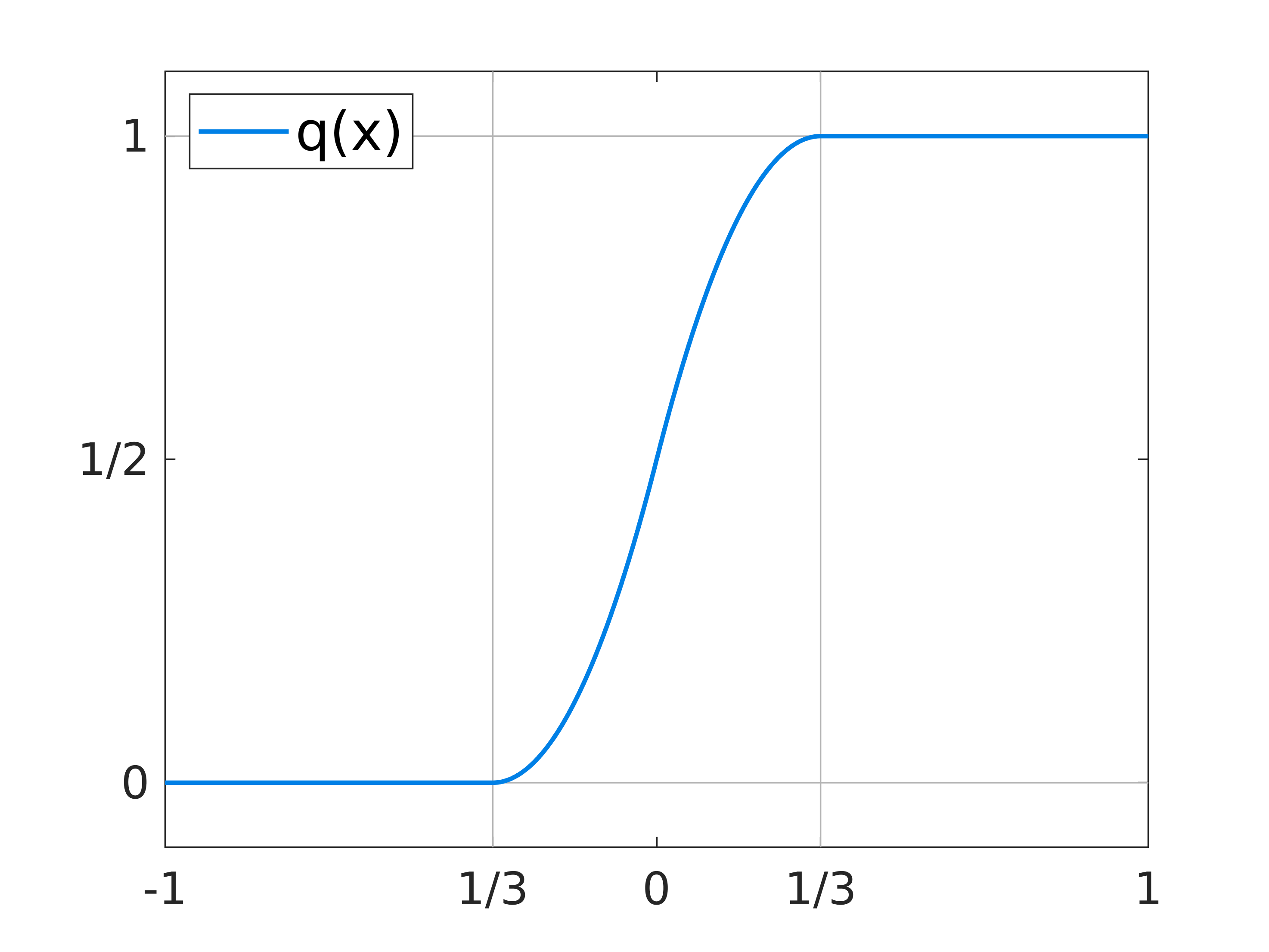}
	\caption{Graphic of the function $q(x)$}
\end{figure}

Moreover, for any $\alpha>0$, let us consider the function $F_\alpha:\mathbb R\to\mathbb R$ given by
\begin{align}\label{F_alpha}
F_\alpha(x) = F(\alpha x):= \sin\left(\alpha x+\frac{(1-s)\pi}{4}\right)-G(\alpha x),
\end{align}
where $G$ is the Laplace transform of the function
\begin{align*}
	\gamma(y):=\frac{\sqrt{4s}\sin(s\pi)y^{2s}}{2\pi(1+y^{4s}-2y^{2s}\cos(s\pi))}\exp\left(\frac 1\pi\int_0^{+\infty}\frac{1}{1+r^2}\log\left(\frac{1-r^{2s}y^{2s}}{1-r^2y^2}\right)\,dr\right).
\end{align*}
According to \cite{kwasnicki2010spectral}, $G$ is a completely monotone function satisfying 
\begin{align}\label{G_bound}
G(\xi)\leq \frac{C(1-s)}{\sqrt{s}} \xi^{-1-2s}, \;\;\textrm{ for all }\;\; \xi\in(0,+\infty).
\end{align}
Then, if we define 
\begin{align}\label{mu_k_def}
\mu_k:=\frac{k\pi}{2}-\frac{(1-s)\pi}{4},\quad k\geq 1,
\end{align}
according to \cite[Example 6.1]{kwasnicki2010spectral} we have that $F_{\mu_k}$ is the solution to the equation
\begin{align*}
\begin{cases}
\fl{s}{F_{\mu_k}(x)}=\mu_k F_{\mu_k}(x) & x> 0,
\\
F_{\mu_k}(x)=0 & x\le 0.
\end{cases}
\end{align*}

In other words, $\{F_{\mu_k}\}_{k\geq 1}$ are the eigenfunctions of the fractional Laplacian on the half-line with the zero exterior Dirichlet condition, and $\{\mu_k\}_{k\geq 1}$ are the corresponding eigenvalues.

Let us now define 
\begin{align}\label{rho_k_def}
	\varrho_k(x):=q(-x)F_{\mu_k}(1+x)+(-1)^k F_{\mu_k}(1-x),\quad k\geq 1.
\end{align}

Notice that $F_{\mu_k}(1+x)=0$ for $x\leq -1$ and $F_{\mu_k}(1-x)=0$ for $x\geq 1$. From this fact, and from the definition \eqref{q_def} of the function $q$, it immediately follows that, for all $k\geq 1$, $\varrho_k(x) = 0$ for $x\in(-1,1)^c$. In addition, it has been proved in \cite{kwasnicki2012eigenvalues} that  there is a constant $C>0$ such that
\begin{align*}
\left| \fl{s}{\varrho_k}(x)-\mu_k^{2s}\varrho_k(x)\right|\leq \frac{C(1-s)}{\sqrt{s}}\mu_k^{-1},\;\;\textrm{ for all }\;\; x\in (-1,1),\;k\geq 1.
\end{align*}

In particular, the family $\{(\varrho_k(x),\mu_k^{2s})\}_{k\geq 1}$ is at a distance $\mathcal O(1/k)$ from the spectrum $\{(\phi_k(x),\lambda_k)\}_{k\geq 1}$. Hence, instead of looking for $L^1$-estimates of the eigenfunctions $\phi_k$, we can consider the functions $\varrho_k$ defined in \eqref{rho_k_def}.

First of all, we can trivially check that 
\begin{align*}
	\varrho_k(x) = F_{\mu_k}(1+x)+q(x)f(x)-\sin\left(\mu_k(1+x)+\frac{(1-s)\pi}{4}\right)\chi_{[1,+\infty)},
\end{align*}
where we have denoted 
\begin{align*}
	f(x):=G\big(\mu_k(1+x)\big)+(-1)^kG\big(\mu_k(1-x)\big).
\end{align*}
Moreover, since the open set $\omega\subset(-1,1)$, we have that $1\pm x>0$ for all $x\in\omega$. Therefore,
\begin{align}\label{qf_est}
	\sup_{x\in\omega}|q(x)f(x)|\leq \frac{c(1-s)}{\sqrt{s}}\mu_k^{-1-2s}.
\end{align}
In addition, for $x\in\omega$ we have 
\begin{align}\label{rho_k_2}
	\varrho_k(x) & = F_{\mu_k}(1+x)+q(x)f(x) \notag 
	\\
	&= \sin\left(\mu_k(1+x)+\frac{(1-s)\pi}{4}\right)-G\big(\mu_k(1+x)\big)+q(x)f(x).
\end{align}
From here, if we denote $b_k:=\int_\omega |\varrho_k(x)|\,dx$, we can easily show that 
\begin{align*}
	B:=\inf_{k\geq 1}b_k>0.
\end{align*}
Indeed, from \eqref{G_bound}, \eqref{qf_est} and \eqref{rho_k_2} we have that
\begin{align*}
	b_k \geq \int_\omega\left| \sin\left(\mu_k(1+x)+\frac{(1-s)\pi}{4}\right)-\frac{c(1-s)}{\sqrt{s}}\mu_k^{-1-2s}\right|\,dx.
\end{align*}
Moreover, since $\lim_{k\to+\infty} \mu_k^{-1-2s} =0$, we have that there exists $k_0\in\NN$ such that
\begin{align*}
	b_k \geq \int_\omega \left|\frac 12 \sin\left(\mu_k(1+x)+\frac{(1-s)\pi}{4}\right)\right|\,dx, \;\;\textrm{ for all }\;\; k>k_0.  
\end{align*}

It follows that $B_1:=\inf_{k> k_0}b_k>0$. Hence, $B>0$ since $b_k>0$ for all $k$ (the integrand being positive except possibly for a set of zero measure in which it is zero). The proof is then concluded.
\end{proof}

\begin{remark}
{\em We mention that it has been shown in \cite[Equation (5.3)]{BFR} that the first positive eigenfunction $\phi_1$ of the fractional Laplacian  satisfies $\phi_1(x)\simeq \big(\mbox{dist}(x,(-1,1)^c)\big)^s=(1-|x|)^s$, in the sense that, there are two positive constants $0<C_1\le C_2$ such that
\begin{align}\label{pos-eig}
	C_1(1-|x|)^s\le \phi_1(x)\le C_2(1-|x|)^s,\;\;\;x\in (-1,1).
\end{align}
}
\end{remark}

\begin{proof}[\bf Proof of Theorem \ref{obs_thm}]
The proof of the estimate \eqref{obs} is based on standard spectral techniques. First of all, it is immediate to show that the solution $p$ of \eqref{frac_heat_adj} can be expressed in the basis of the eigenfunctions $\{\phi_k\}_{k\geq 1}$ of $(-d_x^2)_D^s$ as follows:
\begin{align}\label{phi_eigen}
	p(x,t) = \sum_{k\geq 1} a_k\phi_k(x)e^{-\lambda_k(T-t)},\quad a_k:=\int_{-1}^1 p_T(x)\phi_k(x)\,dx.
\end{align} 

From \eqref{phi_eigen}, by using the orthonormality of the eigenfunctions in $L^2(-1,1)$ and the change of variables $T-t\mapsto t$, it is easy to see that the observability inequality \eqref{obs} can be rewritten as
\begin{align}\label{obs_spectrum}
	\sum_{k\geq 1}|a_k|^2e^{-2\lambda_k T} \leq \C(T) \left(\int_0^T\int_\omega \bigg|\sum_{k\geq 1} a_k\phi_k(x)e^{-\lambda_k t}\bigg|\,dxdt\right)^2.
\end{align} 

The rest of the proof uses two main ingredients: the lower bound \eqref{L1_bound_eigen} for the $L^1(\omega)$-norm of the eigenfunctions $\phi_k$ and the inequality \eqref{obs2} applied to the spectrum of the fractional Laplacian. To this end, we shall first check that the eigenvalues of the fractional Laplacian satisfy the gap and summability conditions \eqref{spectral_cond1} and \eqref{spectral_cond2}. 

For the gap condition \eqref{spectral_cond1}, we can use \cite[Proposition 3]{kwasnicki2012eigenvalues} ensuring that the eigenvalues $\lambda_k$ are all simple if $s\geq 1/2$. 

Concerning instead the summability in \eqref{spectral_cond2}, according to \cite[Theorem 1]{kwasnicki2012eigenvalues} we have the following asymptotic behavior:
\begin{align}\label{eigen_asympt}
\lambda_k = \left(\frac{k\pi}{2}+\frac{(1-s)\pi}{4}\right)^{2s}+ \mathcal O\left(\frac 1k\right)\;\;\mbox{ as } k\to\infty.
\end{align}

Hence, by means of \eqref{eigen_asympt}, it is immediate to see that \eqref{spectral_cond2} holds if and only if $s> 1/2$. If $s\leq 1/2$, instead, $\sum_{k\geq 1} \lambda_k^{-1}$ behaves as the harmonic series and therefore, it is divergent. 

Summarizing, if $s> 1/2$, then both \eqref{spectral_cond1} and \eqref{spectral_cond2} hold and, therefore, the inequality \eqref{obs2} is true.

Next, let $x\in(-1,1)$ be fixed. Then, using \eqref{obs2} with $c_k:=a_k\phi_k(x)$ we have that there exists a constant $\C(T)$, depending only on $T$, such that
\begin{align}\label{Ma1}
	\sum_{k\geq 1} |a_k\phi_k(x)|e^{-\lambda_k T}\le \C(T)\int_0^T\bigg|\sum_{k\geq 1} a_k\phi_k(x)e^{-\lambda_k t}\bigg|\;dt.
\end{align}
Then, integrating \eqref{Ma1} over $\omega$ and using the estimate \eqref{L1_bound_eigen}, we get that
\begin{align*}
	\beta\sum_{k\geq 1} |a_k|e^{-\lambda_k T}\le\int_{\omega}\sum_{k\geq 1} |a_k\phi_k(x)|e^{-\lambda_k T}\;dx\le \C(T)\int_{\omega}\int_0^T\bigg|\sum_{k\geq 1} a_k\phi_k(x)e^{-\lambda_k t}\bigg|\,dtdx.
\end{align*}
It follows from the preceding estimate that
\begin{align}\label{est_prel}
	\beta^2\left(\sum_{k\geq 1} |a_k|e^{-\lambda_k T}\right)^2\le \C(T)^2\left(\int_{\omega}\int_0^T\bigg|\sum_{k\geq 1} a_k\phi_k(x)e^{-\lambda_k t}\bigg|\,dtdx\right)^2.
\end{align}
Since
\begin{align*}
	\sum_{k\geq 1} |a_k|^2e^{-2\lambda_kT} \leq  \left(\sum_{k\geq 1} |a_k|e^{-\lambda_k T}\right)^2,
\end{align*}
it follows from \eqref{est_prel} and Fubini's Theorem that
\begin{align*}
	\sum_{k\geq 1} |a_k|^2e^{-2\lambda_kT} \leq \frac{\C(T)^2}{\beta^2}\left(\int_0^T\int_{\omega}\bigg|\sum_{k\geq 1} a_k\phi_k(x)e^{-\lambda_k t}\bigg|\,dxdt\right)^2.
\end{align*}
We have shown \eqref{obs_spectrum} and the proof is finished.
\end{proof}

\begin{proof}[\bf Proof of Theorem \ref{control_thm_unconstr}]
The proof is based on a duality argument. This approach being classical in PDE control theory, for brevity we are going to present here only the principal ideas. The interested reader may found the complete details in \cite{micu2012time}.

Let us fix $T>0$. For every $p_T\in L^2(-1,1)$, let $p\in L^2((0,T);H_0^s(-1,1))\cap C([0,T];L^2(-1, 1))\cap H^1((0,T);H^{-s}(-1,1))$ be the unique weak solution of the adjoint system \eqref{frac_heat_adj}.

Then, according to \cite[Proposition 2.6]{micu2012time}, for any $z_0\in L^2(-1,1)$ the corresponding solution of \eqref{frac_heat} is null controllable at time $T$ by means of a control function $u\in L^\infty(\omega\times(0,T))$ if and only if the observability inequality \eqref{obs} holds. Moreover, following step by step the proof of that proposition it is easy to obtain the inequality \eqref{control_bound} for the $L^\infty$-norm of the control. Finally, according to \cite[Proposition 2.7]{micu2012time}, the control function $u$ is such that
\begin{align}\label{bang_bang}
	\norm{u}{L^\infty(\omega\times(0,T))} = \norm{p}{L^1(\omega\times(0,T))}.
\end{align}
The proof is then concluded.
\end{proof}

\section{Proof of the main results}\label{pr-main-rel}

In this section, we give the proof of our main results, namely, Theorem \ref{control_thm}. The proof will be divided in three parts.

\subsection{Constrained controllability of the system \eqref{frac_heat}}\label{control_thm_sec}

We present here the proof of the first part of Theorem \ref{control_thm} concerning the controllability of \eqref{frac_heat} through a non-negative control $u\in L^\infty(\omega\times(0,T))$. 

\begin{proof}[\bf Proof of Theorem \ref{control_thm}(I)]
First of all, observe that, since the equation is linear, by subtracting $\widehat{z}$, it is enough to show that there exist a time $T>0$ and a control $v\in L^\infty(\omega\times(0,T))$ fulfilling the constraint $v>-\nu$ a.e. in $\omega\times(0,T)$, such that the solution $\xi$ of the system
\begin{align}\label{frac_heat_xi}
	\begin{cases}
		\xi_t+\fl{s}{\xi}=v\chi_{\omega\times(0,T)} & \mbox{ in }\; Q,
		\\
		\xi = 0 & \mbox{ in }\; Q^c,
		\\
		\xi(\cdot,0) = z_0(\cdot)-\widehat{z}_0(\cdot) & \mbox{ in }\; (-1,1),
	\end{cases}
\end{align}
satisfies $\xi(x,T)=0$ a.e. in $(-1,1)$.

According to Theorem \ref{obs_thm}, the null-controllability of \eqref{frac_heat_xi} with $v\in L^\infty(\omega\times(0,T))$  is equivalent to the observability inequality \eqref{obs}. Actually, the controllability (and therefore, the observability of the adjoint system) being true for any time interval $(\tau,T)$, we also have
\begin{align*}
	\norm{p(\cdot,\tau)}{L^2(-1,1)}^2 \leq \C(T-\tau)\left(\int_\tau^T\int_\omega |p(x,t)|\,dxdt\right)^2. 
\end{align*} 

Using \eqref{phi_eigen}, the fact that the eigenvalues $\{\lambda_k\}_{k\ge 1}$ of the operator $(-d_x^2)_D^s$ form a non-decreasing sequence, and the dissipativity of the fractional heat semigroup ensuring exponential stability, we can readily check that
\begin{align*}
	\norm{p(\cdot,0)}{L^2(-1,1)}^2\leq e^{-2\lambda_1\tau}\norm{p(\cdot,\tau)}{L^2(-1,1)}^2
\end{align*}
for every $0<\tau<T$ and therefore, 
\begin{align*}
	\norm{p(\cdot,0)}{L^2(-1,1)}^2 \leq e^{-2\lambda_1\tau}\C(T-\tau)\left(\int_0^T\int_\omega |p(x,t)|\,dxdt\right)^2. 
\end{align*} 
By duality this means that the control $v$ can be chosen such that 
\begin{align*}
	\norm{v}{L^\infty(\omega\times(0,T))}^2\leq e^{-2\lambda_1\tau}\C(T-\tau)\norm{z_0-\widehat{z}_0}{L^2(-1,1)}^2.
\end{align*}
Taking $\tau = T/2$ we obtain that
\begin{align*}
	\norm{v}{L^\infty(\omega\times(0,T))}^2\leq e^{-\lambda_1T}\C(T)\norm{z_0-\widehat{z}_0}{L^2(-1,1)}^2.
\end{align*}

Furthermore, we recall that the observability constant $\mathcal C(T)$ is uniformly bounded for any $T>0$. Hence, for $T$ large enough we have
\begin{align*}
	\norm{v}{L^\infty(\omega\times(0,T))}^2< \nu.
\end{align*}

This immediately implies that $v>-\nu$. Therefore, we have shown the existence of a control $v>-\nu$ steering the solution of \eqref{frac_heat_xi} from $z_0-\widehat{z}_0$ to zero in time $T>0$, provided that $T$ is large enough. Consequently, the solution $z$ of \eqref{frac_heat} is controllable to the trajectory $\widehat{z}$ in time $T$. 

Finally, if $z_0\geq 0$, thanks to Lemma \ref{lem} we also have $z(x,t)\geq 0$ for almost every $(x,t)\in (-1,1)\times (0,T)$. The proof is finished.
\end{proof}

\subsection{Positivity of the minimal time for constrained controllability}\label{pos_time_thm_sec}

This section is devoted to the proof of the second part of Theorem \ref{control_thm} that shows that the minimal time $T_{\rm min}$ needed for controlling the system \eqref{frac_heat} with a non-negative control $u\in L^\infty(\omega\times (0,T))$ is necessarily strictly positive.

\begin{proof}[\bf Proof of Theorem \ref{control_thm}(II)]
Let us start by writing the	solution of \eqref{frac_heat} in the basis of the eigenfunctions $\{\phi_k\}_{k\geq 1}$, that is, 
\begin{align}\label{sol_z_spect}
	z(x,t) = \sum_{k\geq 1} z_k(t)\phi_k(x),
\end{align}
with 
\begin{align}\label{z_k}
	z_k(t):=\int_{-1}^1 z(x,t)\phi_k(x)\,dx.
\end{align}

Derivating \eqref{z_k} and using \eqref{frac_heat} (or multiplying \eqref{frac_heat} by $\phi_k$ and integrating over $(-1,1)$), we can readily check that the coefficients $z_k(t)$ satisfy the following first order ODE:
\begin{align*}
	\begin{cases}
		z_k'(t) = - \lambda_kz_k(t) + u_k(t), & t\in (0,T)
		\\[10pt]
		\displaystyle z_k(0) = \int_{-1}^1 z_0(x)\phi_k(x)\,dx=:z_k^0,
	\end{cases}
\end{align*}
where we have denoted
\begin{align*}
	u_k(t):=\int_\omega u(x,t)\phi_k(x)\,dx.
\end{align*}
Hence, employing the variation of constants formula we easily get that
\begin{align}\label{z_k_expl}
	z_k(t) = z_k^0e^{-\lambda_k t}+\int_0^te^{-\lambda_k(t-\tau)}u_k(\tau)\,d\tau.
\end{align}
On the other hand, since $z(x,T)=\widehat{z}(x,T)$ a.e. in $(-1,1)$, we have that
\begin{align*}
	z_k(T) = \int_{-1}^1 \widehat{z}(x,T)\phi_k(x)\,dx =: \zeta_k,
\end{align*}
and from \eqref{z_k_expl} we immediately obtain that
\begin{align}\label{z_k_expl2}
	\zeta_k - z_k^0e^{-\lambda_k T} = \int_0^Te^{-\lambda_k(T-\tau)}u_k(\tau)\,d\tau.
\end{align}

Let us denote $u_k^+(t):= \max(u_k(t),0)$ and $u_k^-(t):= \max(-u_k(t),0)$ the positive and negative part of $u_k(t)$, respectively, so that $u_k(t) = u_k^+(t)-u_k^-(t)$. Notice that $u_k^+(t), u_k^-(t)\geq 0$. Moreover, for every $0\leq \tau\leq T$, we have 
\begin{align*}
	e^{-\lambda_kT} \leq e^{-\lambda_k(T-\tau)} \leq 1.
\end{align*}
Therefore, 
\begin{align}\label{uk_plus_inequality}
	e^{-\lambda_kT} \int_0^T u_k^+(\tau)\,d\tau \leq \int_0^Te^{-\lambda_k(T-\tau)}u_k^+(\tau)\,d\tau \leq \int_0^Tu_k^+(\tau)\,d\tau, 
\end{align}
and
\begin{align}\label{uk_minus_inequality}
e^{-\lambda_kT} \int_0^T u_k^-(\tau)\,d\tau \leq \int_0^Te^{-\lambda_k(T-\tau)}u_k^-(\tau)\,d\tau \leq \int_0^Tu_k^-(\tau)\,d\tau.
\end{align}
Moreover, by using \eqref{z_k_expl2}, we have
\begin{align}\label{uk_plus_identity}
	\zeta_k - z_k^0e^{-\lambda_k T} + \int_0^Te^{-\lambda_k(T-\tau)}u_k^-(\tau)\,d\tau = \int_0^Te^{-\lambda_k(T-\tau)}u_k^+(\tau)\,d\tau,
\end{align}
and
\begin{align}\label{uk_minus_identity}
	\zeta_k - z_k^0e^{-\lambda_k T} - \int_0^Te^{-\lambda_k(T-\tau)}u_k^+(\tau)\,d\tau = -\int_0^Te^{-\lambda_k(T-\tau)}u_k^-(\tau)\,d\tau.
\end{align}
From \eqref{uk_plus_inequality} and \eqref{uk_plus_identity}, we can obtain that
\begin{align}\label{uk_plus_bounds}
	\zeta_k - z_k^0e^{-\lambda_k T} &+ \int_0^Te^{-\lambda_k(T-\tau)}u_k^-(\tau)\,d\tau \notag
	\\
	&\leq\int_0^T u_k^+(\tau)\,d\tau \leq\zeta_ke^{\lambda_k T} - z_k^0 + \int_0^Te^{\lambda_k\tau}u_k^-(\tau)\,d\tau.
\end{align}
On the other hand, \eqref{uk_minus_inequality} and \eqref{uk_minus_identity} yield
\begin{align}\label{uk_minus_bounds}
	\zeta_ke^{\lambda_k T} &- z_k^0 - \int_0^Te^{\lambda_k\tau}u_k^+(\tau)\,d\tau \notag
	\\
	&\leq\int_0^T (-u_k^-(\tau))\,d\tau \leq \zeta_k -z_k^0e^{-\lambda_k T} - \int_0^Te^{-\lambda_k(T-\tau)}u_k^+(\tau)\,d\tau.
\end{align}

Assume by contradiction that, for every $T>0$, there exists a non-negative control function $u^T$ steering $z_0$ to $\widehat{z}(\cdot,T)$ in time $T$, and that $\widehat{z}(\cdot,T)\neq z_0$ (otherwise, the trajectory $z\equiv z_0=\widehat{z}(\cdot,T)$ trivially solves the problem). Then, applying \eqref{uk_plus_bounds} to $(u_k^T)^+$ and taking the limit as $T\to 0^+$ we obtain that 
\begin{align}\label{uk_plus_limit}
	\lim_{T\to 0^+} \int_0^T (u_k^T)^+(\tau)\,d\tau = \zeta_k - z_k^0. 
\end{align}

In the same way, applying \eqref{uk_minus_bounds} to $(u_k^T)^-$ and taking the limit as $T\to 0^+$ we obtain that 
\begin{align}\label{uk_minus_limit}
	\lim_{T\to 0^+} \int_0^T -(u_k^T)^-(\tau)\,d\tau = \zeta_k - z_k^0. 
\end{align}
Then, \eqref{uk_plus_limit} and \eqref{uk_minus_limit} yield
\begin{align*}
	\lim_{T\to 0^+} \int_0^T \Big((u_k^T)^+(\tau)-(u_k^T)^-(\tau)\Big)\,d\tau = \lim_{T\to 0^+} \int_0^T u_k^T(\tau)\,d\tau = 2\zeta_k - 2z_k^0=:\gamma. 
\end{align*}
Since $z_0\in L^2(-1,1)$, we clearly have 
\begin{align*}
	\sum_{k\geq 1}|z_k^0|^2 = \sum_{k\geq 1}\left(\zeta_k^2 - \gamma \zeta_k + \frac{\gamma^2}{4}\right) <+\infty,
\end{align*}
which implies that 
\begin{align}\label{gam=0}
	\lim_{k\to+\infty}\left(\zeta_k^2 - \gamma \zeta_k + \frac{\gamma^2}{4}\right)=0.
\end{align}

Moreover, since $\{\phi_k\}_{k\geq 1}$ is an orthonormal complete system in $L^2(-1,1)$, it follows that $\phi_k\rightharpoonup 0$ (weak convergence) in $L^2(-1,1)$ as $k\to+\infty$. This implies that
\begin{align*}
	\lim_{k\to+\infty}(\widehat{z}(\cdot,T),\phi_k)_{L^2(-1,1)}=\lim_{k\to+\infty}\int_{-1}^1\widehat{z}(x,T)\phi_k(x)\;dx=\lim_{k\to+\infty}\zeta_k=0.
\end{align*}
This identity and \eqref{gam=0} yield $\gamma=0$. Thus, we immediately have that
\begin{align*}
	0 = 2z_k^0-2\zeta_k = 2\int_{-1}^1 (z_0(x)-\widehat{z}(x,T))\phi_k(x)\,dx, \;\;\textrm{ for all }\;\; k\geq 1.
\end{align*}

This is possible if and only if $z_0(x)=\widehat{z}(x,T)$ a.e. in $(-1,1)$, which contradicts our previous assumption. The proof is then concluded.
\end{proof}

\subsection{Constrained controllability in minimal time with measure controls}\label{measure_control_thm_sec}

In this section we give the proof of the third part of Theorem \ref{control_thm} which ensures that constrained controllability of the system \eqref{frac_heat} holds in the minimal time $T_{\rm min}$ with controls in the (Banach) space of the Radon measures $\M$ endowed with the norm
\begin{align*}
	\|\mu\|_{\M}:=\sup\bigg\{&\int_{\omega\times (0,T_{\rm min})}\varphi(x,t)\;d\mu(x,t): 
	\\
	&\varphi\in C(\overline{\omega}\times [0,T_{\rm min}],\RR)\mbox{ and }\; \max_{\overline{\omega}\times [0,T_{\rm min}]}|\varphi|=1\bigg\}.
\end{align*}

We recall that solutions of \eqref{frac_heat} with controls in $\M$ are defined by transposition (see \cite{lions1968problemes}). 

%\begin{definition}\label{def51}
%Given $z_0\in L^2(-1,1)$, $u\in\M$ and $T>0$, we shall say that a function $z\in L^2(Q)$ is a weak solution of \eqref{frac_heat} if the identity
%
%\begin{align*}
%\int_Q\Big(-\varphi_t+\fl{s} \varphi\Big)z\;dxdt-\int_{-1}^1z_0(x)\varphi(x,0)\;dx=\int_{\omega\times (0,T_{\rm min})}\varphi(x,t)\;du(x,t)
%\end{align*}
%holds for every $\varphi\in C_c^2([0,T]\times\RR)$, $\varphi(t,x)=0$ for $(x,t)\in Q^c$ and $\varphi (x,T)=0$ in $(-1,1)$. 
%\end{definition}

\begin{definition}
Given $z_0\in L^2(-1,1)$, $T>0$, and $u\in\mathcal M(\omega\times(0,T))$, we shall say that a function $z\in L^1(Q)$ is a solution of \eqref{frac_heat} defined by transposition if it satisfies the identity
\begin{align}\label{U}
	\int_{\omega\times (0,T)}p(x,t)\;du(x,t)= \langle z(\cdot,T),p_T\rangle - \int_{-1}^1 z_0(x)p(x,0)\,dx,
\end{align}
where, for every $p_T\in L^\infty(-1,1)$, the function $p\in L^\infty(Q)$ is the unique solution of the adjoint system
\begin{align}\label{frac_heat_adj_transp}
	\begin{cases}
		-p_t+\fl{s}{p} = 0 & \mbox{ in }\;  Q,
		\\
		p = 0 & \mbox{ in }\; Q^c,
		\\
		p(\cdot,T) = p_T & \mbox{ in }\; (-1,1).	
	\end{cases}
\end{align}
\end{definition}

The existence of a unique transposition solution $z\in L^1(Q)$ of \eqref{frac_heat} is a consequence of the maximum principle for parabolic equations together with duality and approximation arguments. These arguments being classical, we omit here the technical details. 

\begin{proof}[\bf Proof of Theorem  \ref{control_thm}(III)]
Let us now prove the existence of a measure-valued non-negative control $u\in\M$ realizing the controllability of the system \eqref{frac_heat} exactly in time $T_{\rm min}$. Let us denote 	
\begin{align*}
	T_k:=T_{\rm min}+\frac 1k,\;\,k\geq 1.
\end{align*}

In view of the definition \eqref{t_min} of the minimal control time, there exists a sequence of non-negative controls $\{u^{T_k}\}_{k\geq 1}\subset L^\infty(\omega\times(0,T_k))$ such that the corresponding solutions $\{z^k\}_{k\ge 1}$ of \eqref{frac_heat} with $z^k(x,0) = z_0(x)$ a.e. in $(-1,1)$ satisfy $z^k(x,T_k) = \widehat{z}(x,T_k)$ a.e. in $(-1,1)$ and for every $k\ge 1$. We extend these controls by $\widehat{u}$ on $(T_k,T_{\rm min}+ 1)$ to get a new sequence (that we shall denote again by $u^{T_k}$) in $L^\infty(\omega\times(0,T_{\rm min}+1))$. 

We want to prove that this sequence is bounded in $L^1(\omega\times(0,T_{\rm min}+1))$. To this end, let us assume that the initial datum $p_T$ in \eqref{frac_heat_adj_transp} is positive which, thanks to Lemma \ref{lem}, implies that the corresponding solution $p$ satisfies $p(x,t)\geq\theta>0$ for all $(x,t)\in (-1,1)\times (0,T_{\rm min}+1)$ and for some positive constant $\theta$. Then, \eqref{U} and the positivity of $u^{T_k}$ yield
\begin{align*}
	\theta\norm{u^{T_k}}{L^1(\omega\times(0,T_{\rm min}+1))} &= \theta \int_0^{T_{\rm min}+1}\int_{\omega} u^{T_k}(x,t)\,dxdt 
	\\
	&\leq \int_0^{T_{\rm min}+1}\int_{-1}^1 p(x,t)u^{T_k}(x,t)\,dxdt 
	\\
	&= \langle z(\cdot,T),p_T\rangle - \int_{-1}^1 z_0(x)p(x,0)\,dx \leq M,
\end{align*}
where the last inequality is due to the continuous dependence of the solutions  on the initial data of the direct and adjoint problems.

We have shown that the sequence $\{u^{T_k}\}_{k\geq 1}$ is bounded in $L^1(\omega\times (0,T_{\rm min}+1))$, hence, it is bounded in the space $\mathcal M(\omega\times(0,T_{\rm min}+1))$. Thus, extracting a sub-sequence if necessary, we have that
\begin{align*}
	u^{T_k}  \overset{\ast}{\rightharpoonup} \widetilde{u}  \quad\textrm{ weakly-}\!\ast \textrm{ in } \mathcal M(\omega\times(0,T_{\rm min}+1)) \;\mbox{ as } k\to+\infty.
\end{align*}
Clearly, the limit control $\widetilde u$  satisfies the non-negativity constraint.

Now, for any $k$ large enough and $T_{\rm min}< T_0<T_{\rm min}+1$, by \eqref{U} and the definition of the control $u^{T_k}$ we have
\begin{align}\label{wl}
	\int_{\omega\times (0,T_0)}p(x,t)\;du^{T_k}(x,t)= \langle \widehat{z}(\cdot,T_0),p_T\rangle - \int_{-1}^1 z_0(x)p(x,0)\,dx.
\end{align}

Letting $p_T$ be the first non-negative eigenfunction $\phi_1$ of $(-d_x^2)_D^s$ (see \eqref{pos-eig}), or generally any non-negative function in the domain $D((-d_x^2)_D^s)$ of the operator $(-d_x^2)_D^s$, we get that the solution $p$ of the system \eqref{frac_heat_adj_transp} belongs to $C([0,T]; D((-d_x^2)_D^s))\hookrightarrow C([0,T]\times[-1,1])$.
Therefore, by definition of the $\rm weak^*$ limit, letting $k\to +\infty$ in \eqref{wl} we obtain that
\begin{align*}
	\int_{\omega\times (0,T_0)}p(x,t)\;d\,\widetilde u(x,t)= \langle \widehat{z}(\cdot,T_0),p_T\rangle - \int_{-1}^1 z_0(x)p(x,0)\,dx,
\end{align*}
which in turn implies that $z(x,T_0)=\widehat z(x,T_0)$ a.e. in $(-1,1)$. Finally, taking the limit as $T_0\to T_{\rm min}$ and using the fact that
\begin{align*}
	 |\widetilde u|(\omega\times(T_{\rm min},T_0)) = |\hat u|(\omega\times(T_{\rm min},T_0))=0,\;\;\mbox{ as }\; T_0\to T_{\rm min},
\end{align*}
we can deduce that $z(x,T_{\rm min})=\widehat z(x,T_{\rm min})$ a.e. in $(-1,1)$. This concludes the proof.
\end{proof}

\subsection{Lower bounds for the minimal controllability time}\label{sec44}

Theorem \ref{control_thm} shows that the fractional heat equation \eqref{frac_heat} is controllable to positive trajectories by means of a non-negative control $u$, provided that the controllability time is large enough. Moreover, in the minimal controllability time $T_{\rm min}$ defined by \eqref{t_min}, we proved that the non-negative controls are in the space $\M$ of Radon measures. Notwithstanding, our result does not provide a precise estimate for $T_{\rm min}$.

This is indeed a delicate issue. In \cite{pighin2018controllability}, it has been addressed for the case of the classical linear and semi-linear heat equations by means of a quite general approach. Nevertheless, this methodology does not immediately apply to the fractional heat equation \eqref{frac_heat}. 

In order to clarify this point, in what follows we present an abridged description of how the techniques developed in \cite{pighin2018controllability} should be applied in our case and we highlight the difficulties we encounter.

The starting point is to notice that, by a simple translation argument, we can reduce ourselves to consider the case in which the fractional heat equation \eqref{frac_heat} has zero initial datum. Then, from the definition of transposition solutions of \eqref{U} we have
\begin{align*}
	\langle z(\cdot,T),p_T\rangle-\int_{\omega\times (0,T)}p(x,t)\;du(x,t) = 0.
\end{align*}

Following the procedure of \cite{pighin2018controllability}, the idea is now to find $T_0>0$ and $p_T\in L^2(-1,1)$ such that the corresponding solution of the adjoint system \eqref{frac_heat_adj} satisfies 
\begin{align}\label{p_prop}
	\begin{cases}
		p \geq 0,  & \textrm{ in }\omega\times(0,T_0),
		\\
		\langle\,\widehat{z}(\cdot,T),p_T\rangle<0,  &\textrm{ for all }T\in [0,T_0).
	\end{cases}
\end{align}

Then, an explicit lower bound of the controllability minimal time is obtained by analyzing sharply the conditions required for \eqref{p_prop} to hold. See \cite[Sections 5.1 and 6.1]{pighin2018controllability} for more details.

The choice of a suitable initial datum for the adjoint equation \eqref{frac_heat_adj} is not at all obvious. In \cite{pighin2018controllability}, for the case of the linear and semi-linear heat equations with a boundary control, the authors have proposed to consider $p_T$ in the form
\begin{align}\label{pT1}
	p_T = -\phi_1+2(1-\zeta)\phi_1 
\end{align}
or 
\begin{align}\label{pT2}
	p_T = -\alpha\phi_1+\beta\phi_3,
\end{align}
where $\phi_1$ and $\phi_3$ are respectively the first and third eigenfunction of the Dirichlet Laplacian, $\alpha$ and $\beta$ are suitable positive constants, and $\zeta$ is a cut-off function supported outside the control region.

With these choices, a lower estimate for $T_{\rm min}$ is obtained by employing the positivity of $\phi_1$ and the explicit knowledge of the eigenfunctions $\phi_1$ and $\phi_3$.

Nevertheless, the above proposals for $p_T$ do not seem to be appropriate for our fractional heat equation \eqref{frac_heat}. This, for at least two main reasons. On the one hand, we cannot ensure that with $p_T$ in the form \eqref{pT1} or \eqref{pT2} the corresponding solution of the adjoint equation \eqref{frac_heat_adj} remains positive in $\omega$. On the other hand, even if we were able to overcome this first difficulty, for the eigenfunctions of the Dirichlet fractional Laplacian we do not have a nice expression in terms of sinusoidal functions, as in the case of the classical local operator. Therefore, to perform explicit estimates is a much more difficult issue.

In view of this discussion, we can conclude that the methodology just presented is not immediately applicable to the context of the present paper. For this reason, we are not able to provide explicit analytic lower estimates for the minimal time. This will be done, instead, in Section \ref{numerical_sec} through a numerical approach.

\section{Numerical simulations}\label{numerical_sec}

Our main result, Theorem \ref{control_thm}, states that the fractional heat equation \eqref{frac_heat} is controllable from any initial datum $z_0\in L^2(-1,1)$ to any positive trajectory $\widehat{z}$ by means of the action of a non-negative control $u\in L^\infty(\omega\times(0,T))$, provided that $s>1/2$ and the controllability time $T$ is large enough.

In this section, we present some numerical simulations that confirm these theoretical results. In particular, we will focus on two specific situations:
\begin{itemize}
	\item \textbf{Case 1}: we choose as initial datum 
	\begin{align*}
		z_0(x) := 2\cos\left(\frac \pi2 x\right),
	\end{align*}
	and we set as a target $\widehat{z}(\cdot,T)$ the solution at time $T$ of \eqref{frac_heat} with initial datum 
	\begin{align*}
		\widehat{z}_0(x) := \frac {1}{20}\cos\left(\frac \pi2 x\right)
	\end{align*}
	and right-hand side $\widehat{u}\equiv 1/5$. In this case, as it can be observed in figure \ref{constrained_down_tmin_fig}, we have $z_0>\widehat{z}(\cdot,T)$ a.e. in $(-1,1)$.
	
	\item \textbf{Case 2}: we choose as initial datum 
	\begin{align*}
		z_0(x) := \frac 12\cos\left(\frac \pi2 x\right),
	\end{align*}
	and we set as a target $\widehat{z}(\cdot,T)$ the solution at time $T$ of \eqref{frac_heat} with initial datum 
	\begin{align*}
		\widehat{z}_0(x) := 6\cos\left(\frac \pi2 x\right)
	\end{align*}
	and right-hand side $\widehat{u}\equiv 1$. In this case, as it can be observed in figure \ref{constrained_up_tmin_fig}, we have $z_0<\widehat{z}(\cdot,T)$ a.e. in $(-1,1)$.
\end{itemize}

%\begin{itemize}
%	\item \textbf{Case 1}: the initial datum $z_0$ is above the final target $\widehat z(\cdot,T)$. In particular we choose 
%	\begin{align*}
%		z_0(x) = 2\cos\left(\frac \pi2 x\right),
%	\end{align*}
%	and we set as a target the function $\widehat{z}$ solution of \eqref{frac_heat} with initial datum 
%	\begin{align*}
%		\widehat{z}_0(x) = \frac {1}{20}\cos\left(\frac \pi2 x\right)
%	\end{align*}
%	and right-hand side $\widehat{u}\equiv 1/5$. 
%	\item \textbf{Case 2}: the initial datum $z_0$ is below the final target $\widehat z(\cdot,T)$. In particular we choose 
%	\begin{align*}
%		z_0(x) = \frac 12\cos\left(\frac \pi2 x\right),
%	\end{align*}
%	and we set as a target the function $\widehat{z}$ solution of \eqref{frac_heat} with initial datum 
%	\begin{align*}
%		\widehat{z}_0(x) = 6\cos\left(\frac \pi2 x\right)
%	\end{align*}
%	and right-hand side $\widehat{u}\equiv 1$. 
%\end{itemize}

In both cases, we first estimate numerically $T_{\rm min}$ by formulating the minimal-time control problem as an optimization one. In a second moment, we will show that for $T\geq T_{\rm min}$ the fractional heat equation \eqref{frac_heat} is controllable from $z_0\in L^2(-1,1)$ to the given trajectories $\widehat{z}(\cdot,T)$ (see above) by means of a non-negative control $u$, while for $T< T_{\rm min}$ this controllability result is not achieved.

To simplify the presentation, we always choose the interval $\omega=(-0.3,0.8)\subset (-1,1)$ as the control region. Moreover, we focus on the case $s>1/2$, in which we know that \eqref{frac_heat} is controllable. We recall that, if $s\leq 1/2$, it has been shown in \cite{biccari2017controllability}, both on a theoretical and numerical level, that the fractional heat equation \eqref{frac_heat} is not controllable, not even in the absence of constraints.

\subsection{Case 1: $z_0>\widehat{z}(\cdot,T)$}

Let us start by analyzing the case of an initial datum $z_0$ above the final target $\widehat z(\cdot,T)$. As we have mentioned above, we first estimate the minimal controllability time $T_{\rm min}$ through a suitable optimization problem and we address the numerical constrained controllability of \eqref{frac_heat} in a time horizon $T\geq T_{\rm min}$. In a second moment, we will consider the case $T< T_{\rm min}$.

\subsubsection{Numerical approximation of the minimal controllability time}

To obtain an approximation of the minimal controllability time $T_{\rm min}$, we solve the following constrained optimization problem:
\begin{align}\label{Topt}
	\textrm{minimize}\;T
\end{align}
subject to the constraints
\begin{align}\label{Topt_constr}
	& T>0, \notag
	\\
	&z_t+\fl{s}{z}=u\chi_\omega, & &\textrm{a.e. in } (-1,1)\times(0,T), \notag
	\\
	&z(x,0)=z_0\geq 0, & &\textrm{a.e. in } (-1,1),
	\\
	&z(x,t)\geq 0, & &\textrm{a.e. in } (-1,1)\times(0,T), \notag
	\\
	&u(x,t)\geq 0, & &\textrm{a.e. in } \omega\times(0,T). \notag
\end{align}

To solve this problem numerically, we employ the expert interior-point optimization routine \texttt{IpOpt} (see \cite{wachter2006implementation}) combined with automatic differentiation and the modeling language \texttt{AMPL} (\cite{fourer1990modeling}).

To perform our simulations, we apply a FE method for the space discretization of the fractional Laplacian on a uniform space-grid $x_i = -1 + \frac{2i}{N_x}$, $i = 1,\ldots,N_x$, with $N_x=20$ (see \cite{biccari2017controllability}). Moreover, we use an explicit Euler scheme for the time integration on the time-grid $t_j = \frac{Tj}{N_t}$,  $j = 0,\ldots,N_t$, with $N_t$ satisfying the Courant-Friedrich-Lewy condition. In particular, we will choose here $N_t=300$.

The minimal time that we obtain from our simulations is $T_{\rm min}\simeq 0.8285$ and we can see in figure \ref{constrained_down_tmin_fig} that, in this time horizon, we are able to steer the initial datum $z_0$ to the desired target by maintaining the positivity of the solution.

\begin{figure}[h]
	\centering
	\includegraphics[scale=0.45]{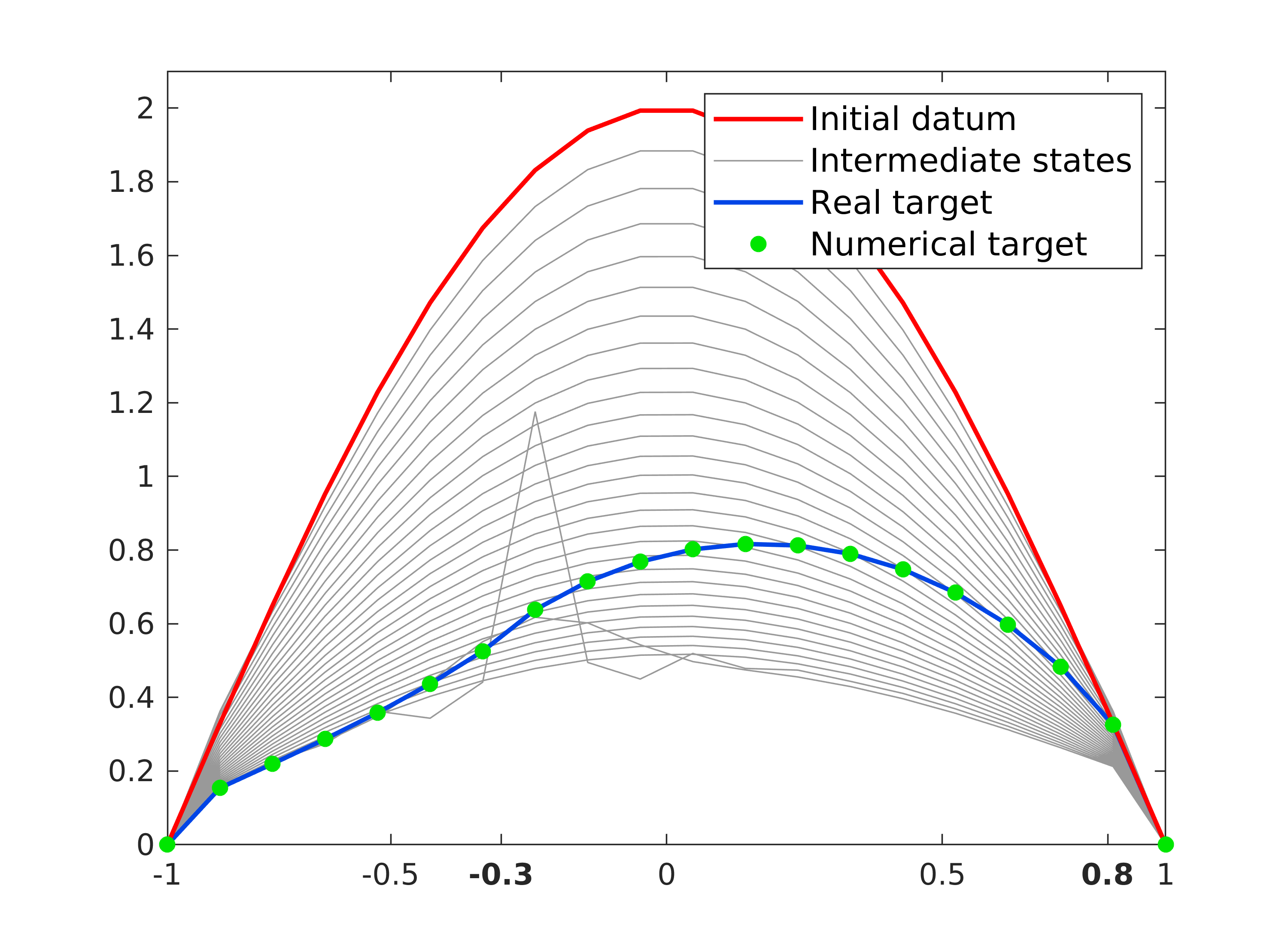}
	\caption{Evolution in the time interval $(0,T_{\rm min})$ of the solution of \eqref{frac_heat} with $s=0.8$. The blue curve is the target we want to reach while the green bullets indicate the target we computed numerically.}\label{constrained_down_tmin_fig}
\end{figure}

We have to stress here that the minimal time $T_{min}$ we have obtained is just an approximation computed by solving numerically the optimization problem \eqref{Topt}-\eqref{Topt_constr}. The validity of this approximation shall be confirmed by a convergence result as the mesh-sizes tend to zero. We will present more details on this issue in Section \ref{sec-con-rem}.

In figures \ref{contol_down_tmin_disposition_fig} and \ref{contol_down_tmin_fig}, we show the behavior of the minimal-time control corresponding to the dynamics of figure \ref{constrained_down_tmin_fig}. As we can see, the control is initially inactive and leaves the equation evolving under the dissipative effect of the heat semigroup. When the state finally approaches the target, the control prevents it to pass below by means of an impulsional action localized in certain specific points of the control region. Moreover, we have to mention that, since the range of amplitudes of the impulses of our minimal-time control is quite large, the plot in figure \ref{contol_down_tmin_fig} (and in figure \ref{contol_up_tmin_fig} below) is in logarithmic scale. In this way,  also the impulses with smaller amplitude can be appreciated. 

\begin{figure}[h]
	\centering
	\includegraphics[scale=0.2]{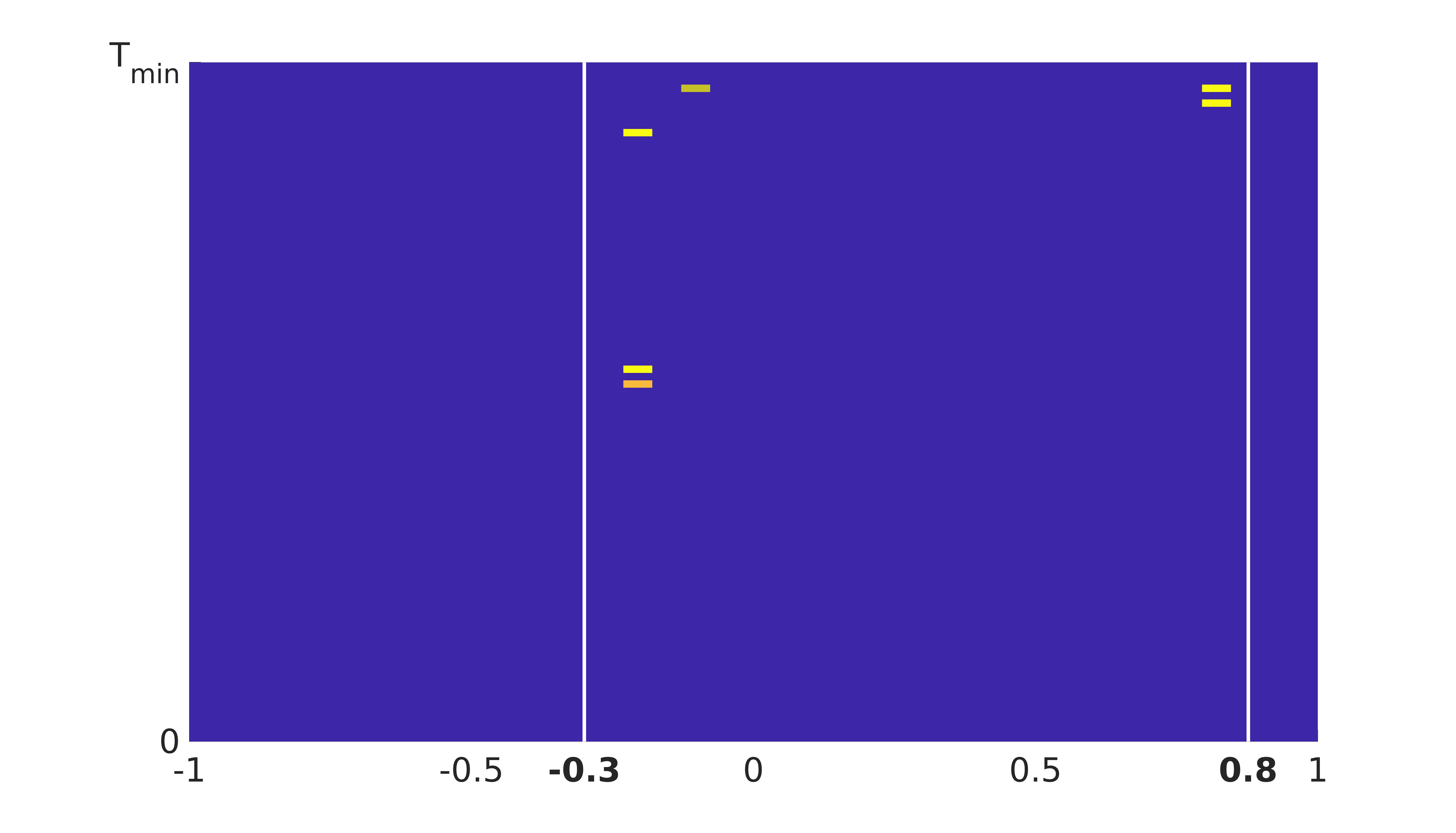}
	\caption{Minimal-time control: space-time distribution of the impulses. The white lines delimit the control region $\omega=(-0.3,0.8)$. The regions in which the control is active are marked in yellow.}\label{contol_down_tmin_disposition_fig}
\end{figure}

\begin{figure}[h]
	\centering
	\includegraphics[scale=0.25]{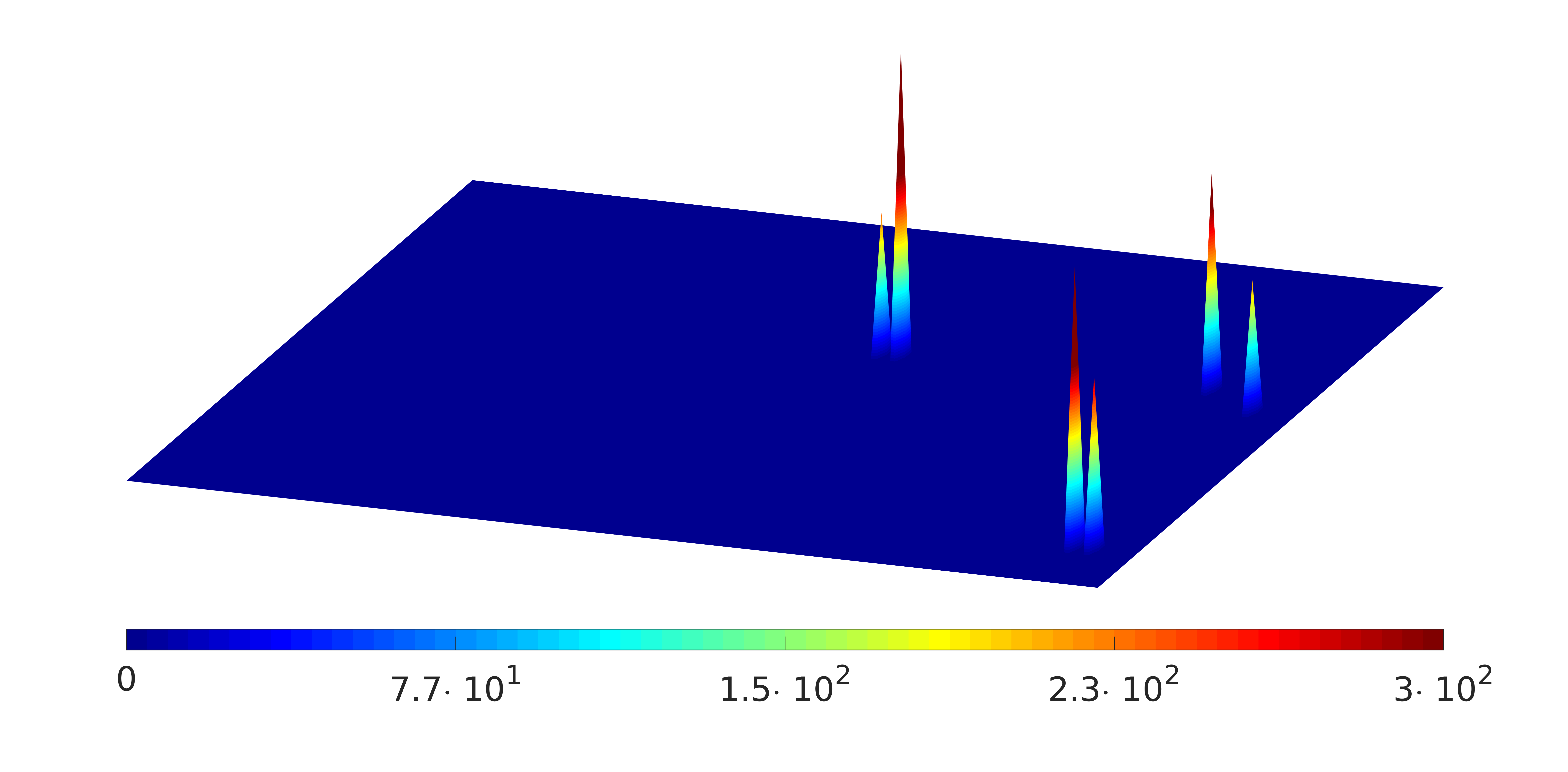}
	\caption{Minimal-time control: intensity of the impulses in logarithmic scale. In the $(t,x)$ plane in blue the time $t$ varies from $t=0$ (left) to $t=T_{\rm min}$ (right).}\label{contol_down_tmin_fig}
\end{figure}

This behavior of the control is not surprising. Indeed, as it was already observed in \cite{loheac2017minimal,pighin2018controllability}, the minimal-time controls are expected to be \textit{atomic} measures, in particular linear combinations of Dirac deltas. Our simulations are thus consistent with the aforementioned papers. Nevertheless, a more complete analysis of the positions and amplitudes of these impulses shall be addressed. This will be the subject of a future work.

The impulsional behavior of the control is then lost when extending the time horizon beyond $T_{\rm min}$. In figure \ref{constrained_down_tlarge_fig}, we show the evolution of the solution of the fractional heat equation \eqref{frac_heat} from the initial datum $z_0$ to the target $\widehat{z}(\cdot,T)$ in the time horizon $T=0.9$. As we can observe, in accordance with our theoretical results, the equation is still controllable in time $T$. Nevertheless, the action of the control is now more distributed in $\omega$ (see figure \ref{control_down_tlarge_fig}). 

\begin{figure}[h]
	\includegraphics[scale=0.45]{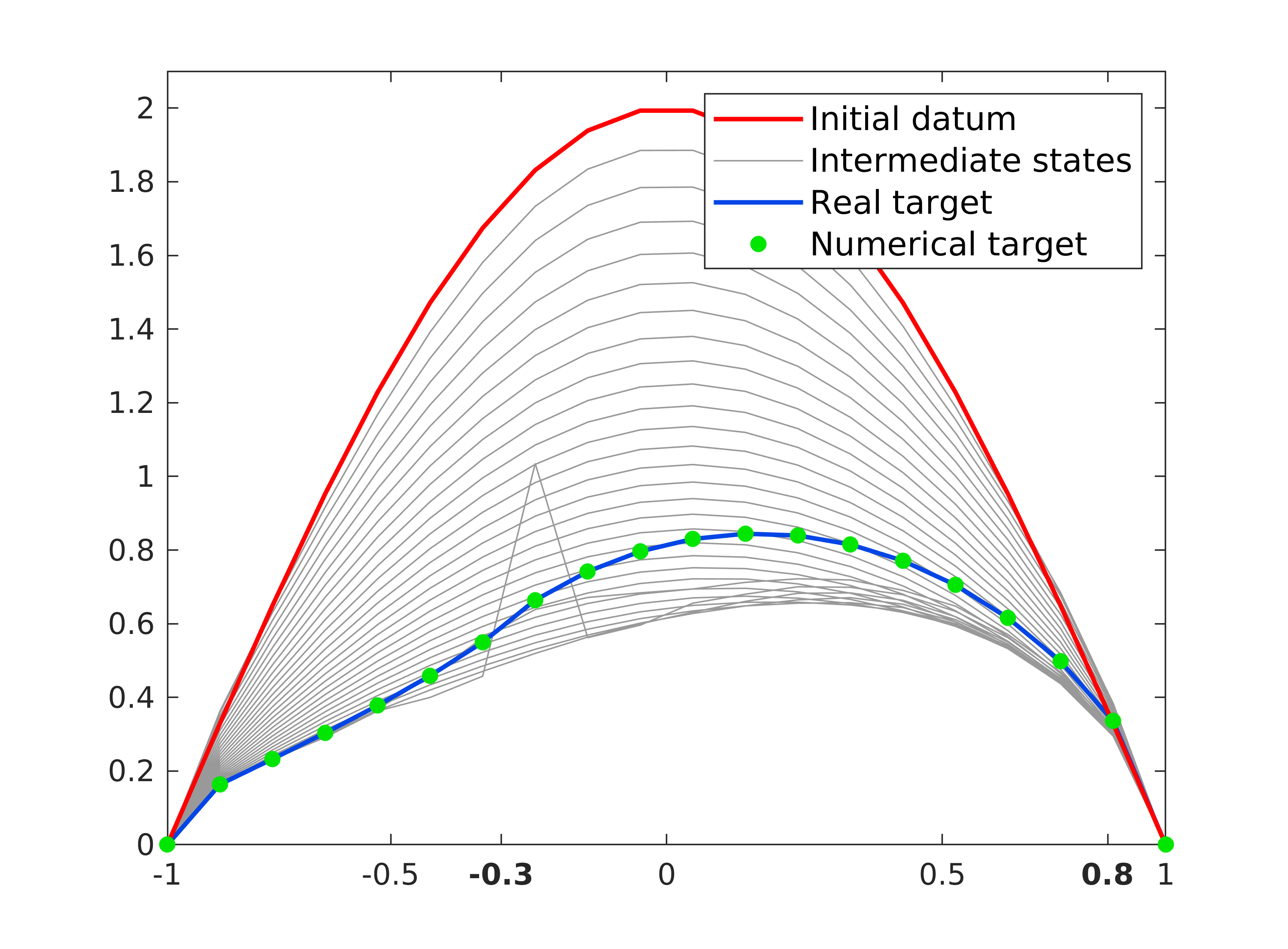}
	\caption{Evolution in the time interval $(0,0.9)$ of the solution of \eqref{frac_heat} with $s=0.8$. The blue curve is the target we want to reach while the green bullets indicate the target we computed numerically.}\label{constrained_down_tlarge_fig}
\end{figure}

\begin{figure}[h]
	\centering
	\includegraphics[scale=0.2]{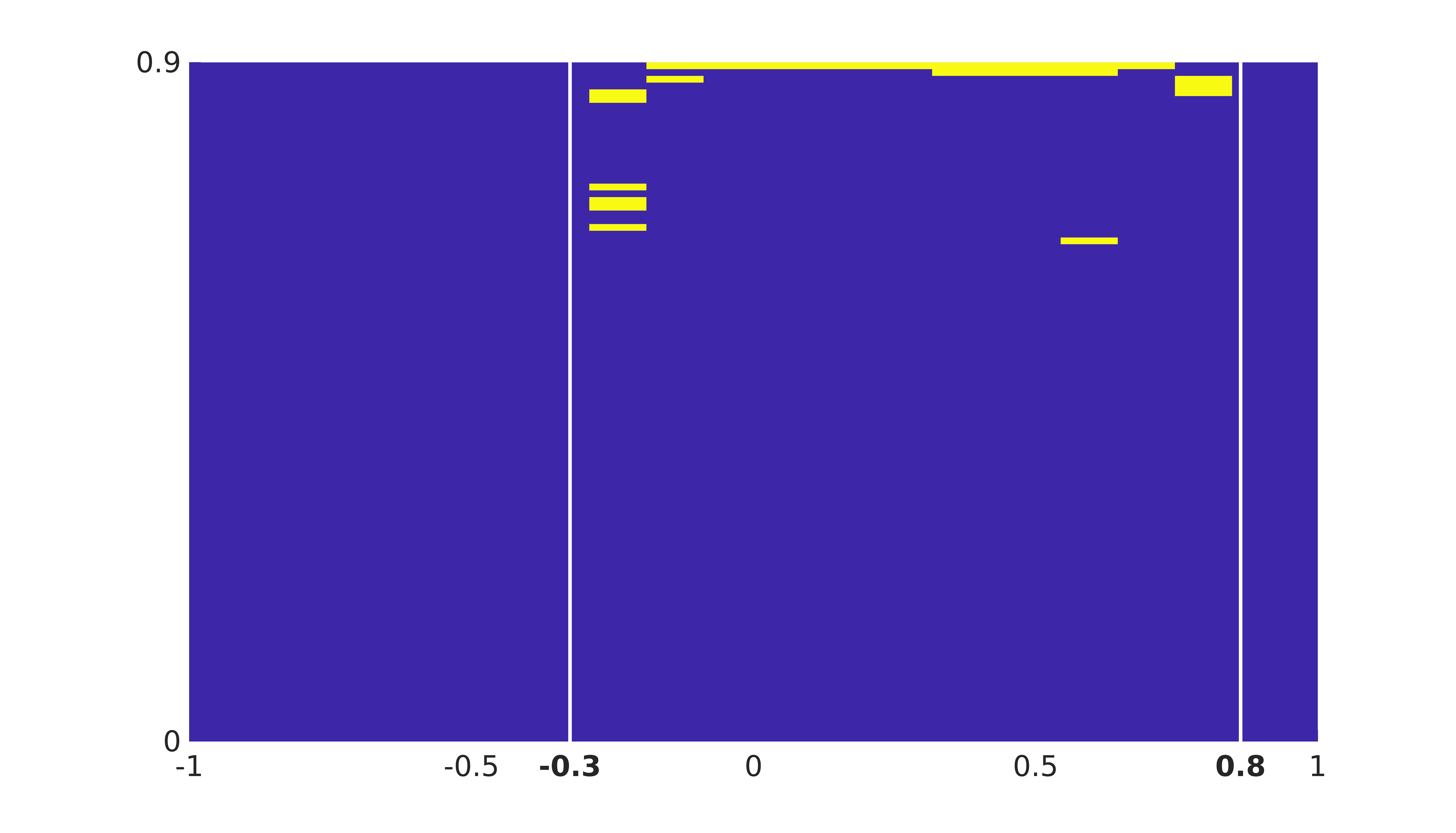}
	\caption{Behavior of the control in time $T=0.9$. The white lines delimit the control region $\omega=(-0.3,0.8)$. The regions in which the control is active are marked in yellow. The atomic nature is lost.}\label{control_down_tlarge_fig}
\end{figure}

\subsubsection{Lack of controllability in short time}

In this section, we conclude our discussion on Case 1 by showing the lack of controllability of \eqref{frac_heat} when the 
time horizon is too short. 

To this end, we employ a classical conjugate gradient method implemented in the DyCon Computational Toolbox (\cite{dycontoolbox}) for solving the following optimization problem:

\begin{align}\label{opt_dycon}
	\min\;\norm{z(\cdot,T)-\widehat{z}(\cdot,T)}{L^2(-1,1)}
\end{align}
subject to the constraints
\begin{align}\label{opt_dycon_constr}
	&z_t+\fl{s}{z}=u\chi_\omega, & & \textrm{a.e. in } (-1,1)\times(0,T), \notag
	\\
	&z(x,0)=z_0\geq 0, & &\textrm{a.e. in } (-1,1),
	\\
	&z(x,t)\geq 0, & & \textrm{a.e. in } (-1,1)\times(0,T), \notag
	\\
	&u(x,t)\geq 0, & &\textrm{a.e. in } \omega\times(0,T). \notag
\end{align}

As before, we apply a FE method for the space discretization of the fractional Laplacian on a uniform space-grid with $N_x=20$ points and we use an explicit Euler scheme for the time integration on a time-grid with $N_t=300$ points. Furthermore, we choose a time horizon $T=0.7$, which is below the minimal controllability time $T_{\rm min}$. 

Our simulation then show that the solution of \eqref{frac_heat} fails to be controlled (see figure \ref{constrained_down_tsmall_fig}). In fact, since we want to reach a final target which is below the initial datum $z_0$, the natural approach would be to push down the state with a ``negative'' action. Since, however, the control is not allowed to do this because of the non-negativity constraint, its best option is to remain inactive for the entire time interval and to let the solution diffuses under the action of the fractional heat semigroup. Nevertheless, this is not enough to reach the target in the time horizon provided.

\begin{figure}[h]
	\begin{minipage}{0.47\textwidth}
		\centering
		\includegraphics[scale=0.43]{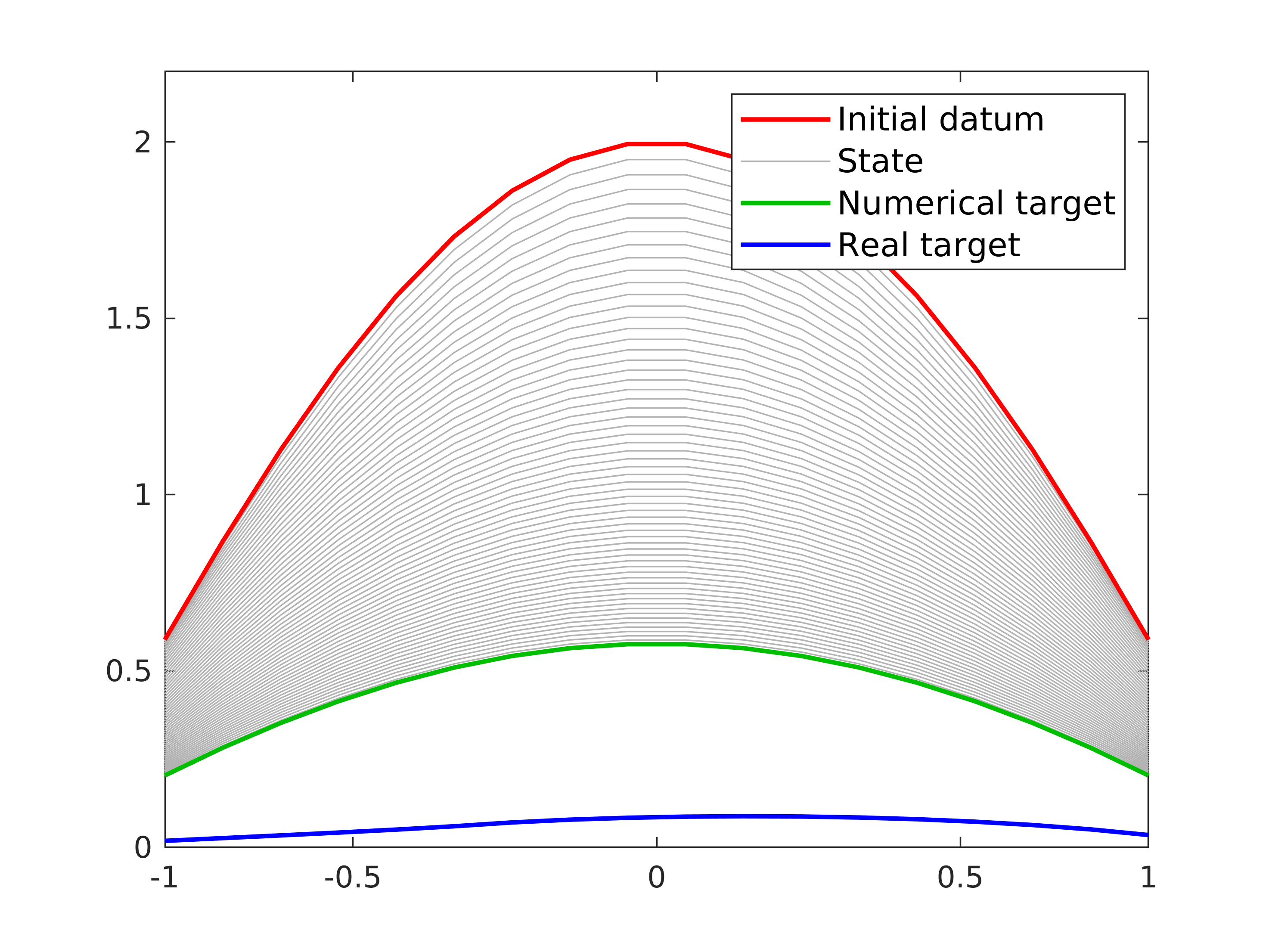}
	\end{minipage}
	\begin{minipage}{0.47\textwidth}
		\centering
		\includegraphics[scale=0.43]{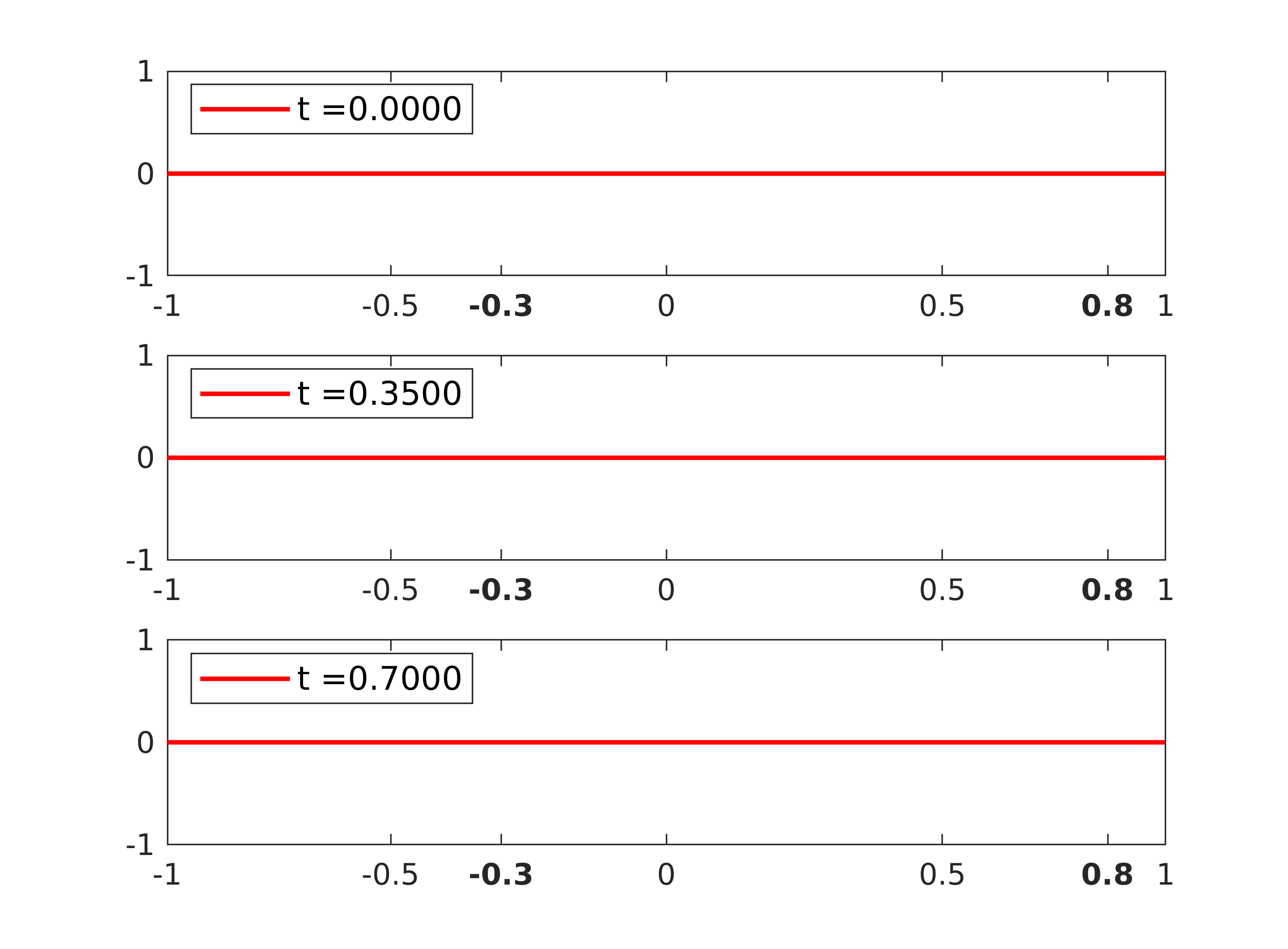}
	\end{minipage}
	\caption{Evolution in the time interval $(0,0.7)$ of the solution of \eqref{frac_heat} with $s=0.8$ (left) and of the control $u$ (right), under the constraint $u\geq 0$. The bold characters highlight the control region $\omega=(-0.3,0.8)$. The control remains inactive during the entire time interval, and the equation is not controllable.}\label{constrained_down_tsmall_fig}
\end{figure}

\subsection{Case 2: $z_0<\widehat{z}(\cdot,T)$}

Let us now consider the case of an initial datum $z_0$ which is smaller than the final target $\widehat{z}(\cdot,T)$. As before, we start by using \texttt{IpOpt} for solving the optimization problem \eqref{Topt}-\eqref{Topt_constr} and computing the minimal controllability time.

Also in this case, we apply a FE method for the space discretization of the fractional Laplacian on a uniform space-grid with $N_x=20$ points and we use an explicit Euler scheme for the time integration on a time-grid with $N_t=100$ points. 

This time, we obtain $T_{\rm min}\simeq 0,2101$ and, once again, our simulations displayed in figure \ref{constrained_up_tmin_fig} show that in this time horizon the fractional heat equation \eqref{frac_heat} is controllable from the initial datum $z_0$ to the desired trajectory $\widehat{z}(\cdot,T)$.
 
\begin{figure}[h]
	\centering
	\includegraphics[scale=0.45]{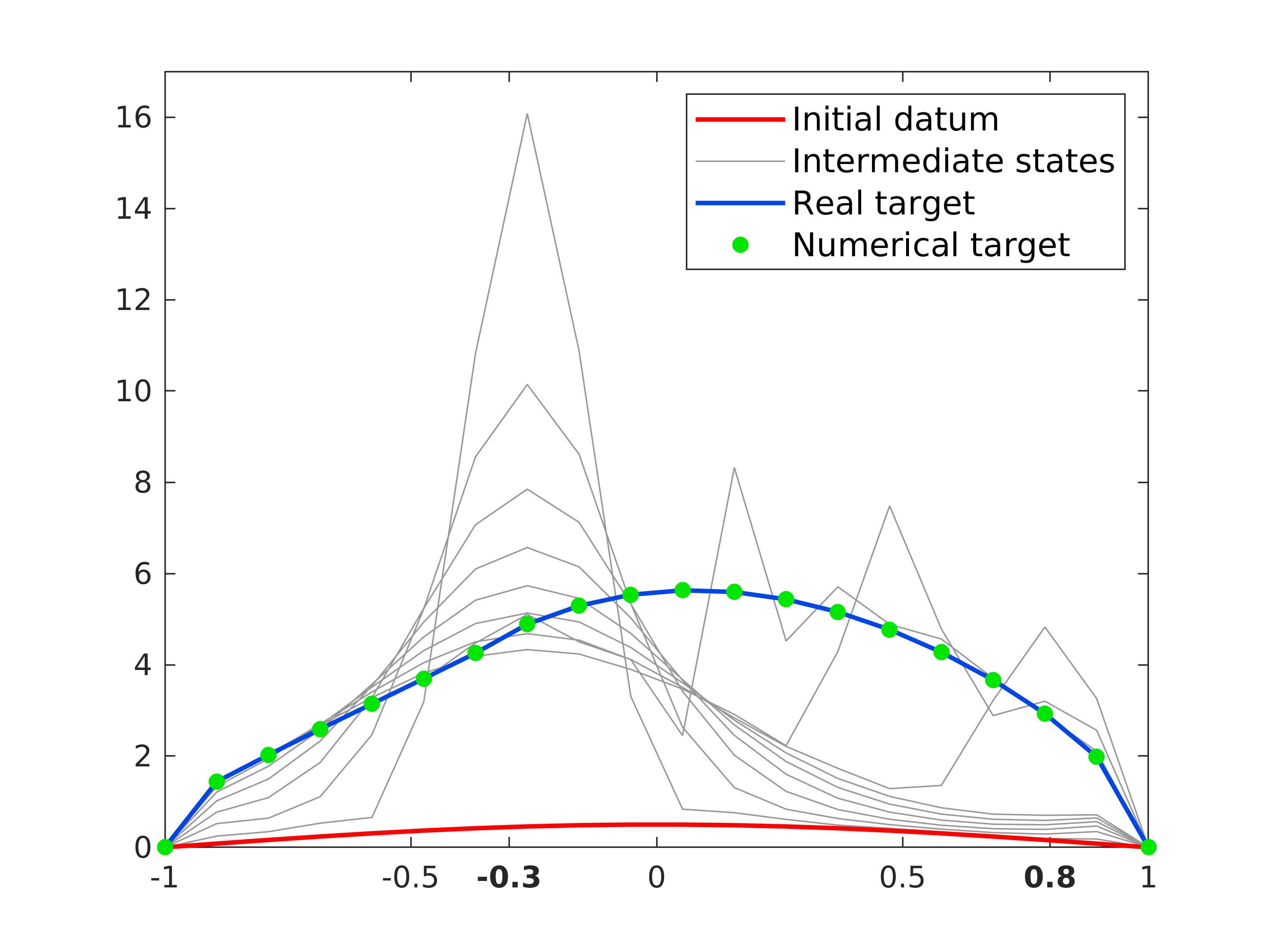}
	\caption{Evolution in the time interval $(0,T_{\rm min})$ of the solution of \eqref{frac_heat} with $s=0.8$. The blue curve is the target we want to reach while the green bullets indicate the target we computed numerically.}\label{constrained_up_tmin_fig}
\end{figure}

Nevertheless, notice that, this time, we want to reach a target which is above the initial datum $z_0$. This means that the control needs to countervail also the dissipation of the solution of \eqref{frac_heat}, by acting on it from the very beginning with a positive force. Moreover, also in this case, the minimal-time control has an atomic nature, as it is shown in figures \ref{contol_up_tmin_disposition_fig} and \ref{contol_up_tmin_fig}.

\begin{figure}[h]
	\centering
	\includegraphics[scale=0.18]{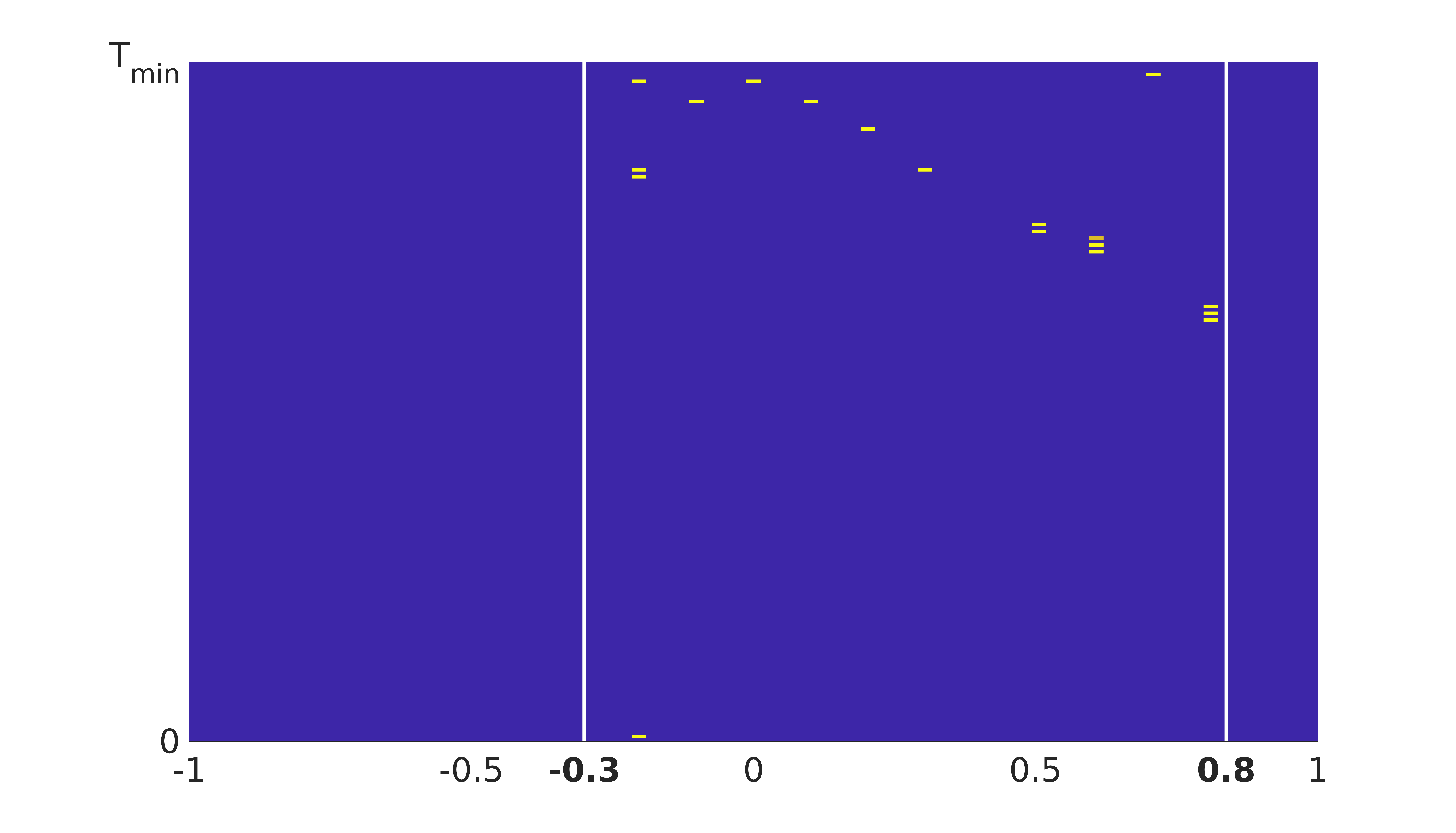}
	\caption{Minimal-time control: space-time distribution of the impulses. The white lines delimit the control region $\omega=(-0.3,0.8)$. The regions in which the control is active are marked in yellow.}\label{contol_up_tmin_disposition_fig}
\end{figure}

\begin{figure}[h]
	\centering
	\includegraphics[scale=0.22]{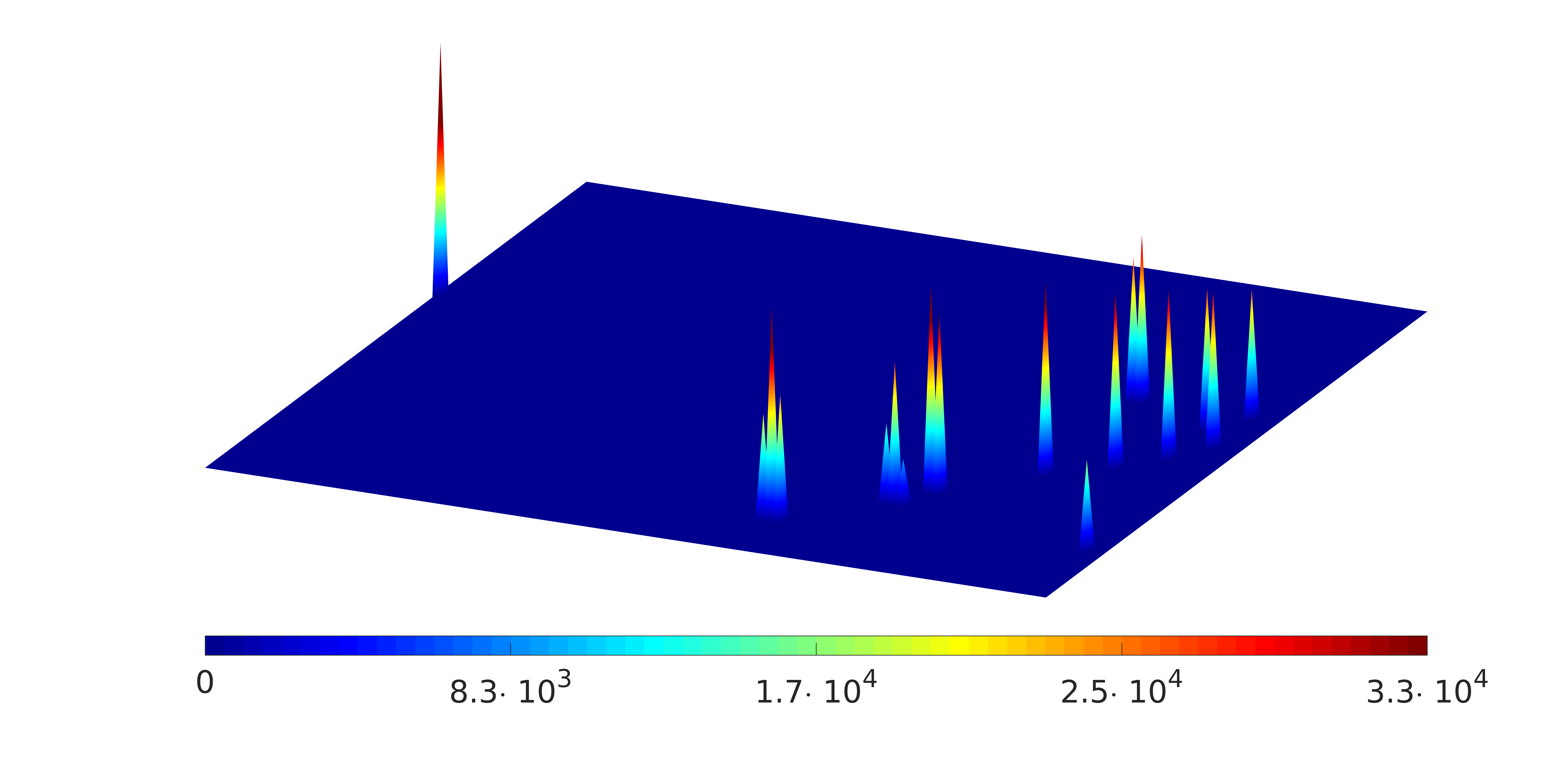}
	\caption{Minimal-time control: intensity of the impulses in logarithmic scale. In the $(t,x)$ plane in blue the time $t$ varies from $t=0$ (left) to $t=T_{\rm min}$ (right).}\label{contol_up_tmin_fig}
\end{figure}

Moreover, when extending the time horizon beyond $T_{\rm min}$ we can observe once again how the solution of \eqref{frac_heat} is still controlled but, this time, the control is distributed in a larger part of the control region $\omega$ and not anymore localized in specific points (see Figures \ref{constrained_up_tlarge_fig} and \ref{control_down_up_fig}).

\begin{figure}[h]
	\centering
	\includegraphics[scale=0.43]{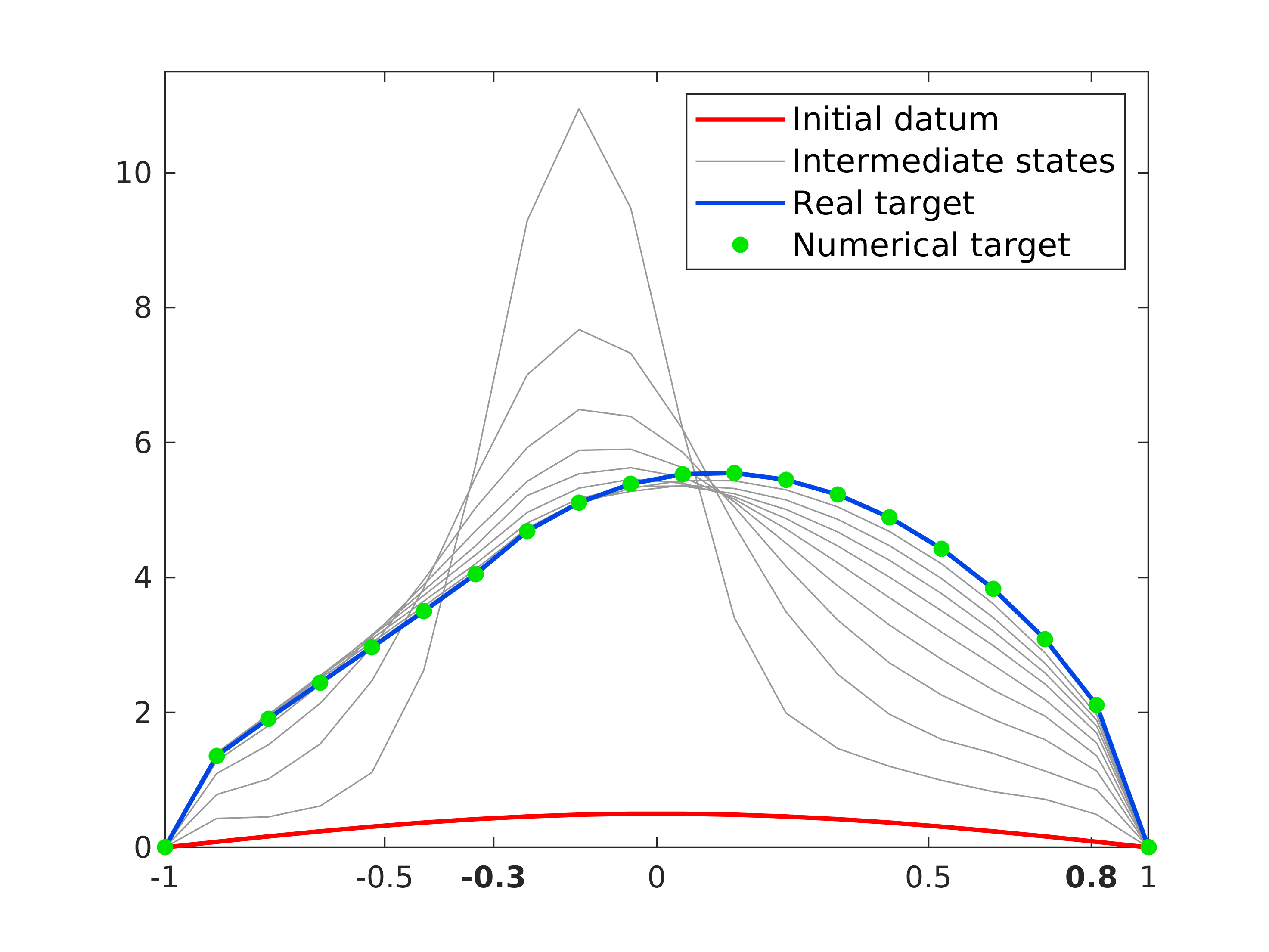}
	\caption{Evolution in the time interval $(0,0.4)$ of the solution of \eqref{frac_heat} with $s=0.8$. The blue curve is the target we want to reach while the green bullets indicate the target we computed numerically.}\label{constrained_up_tlarge_fig}
\end{figure}

\begin{figure}[h]
	\centering
	\includegraphics[scale=0.18]{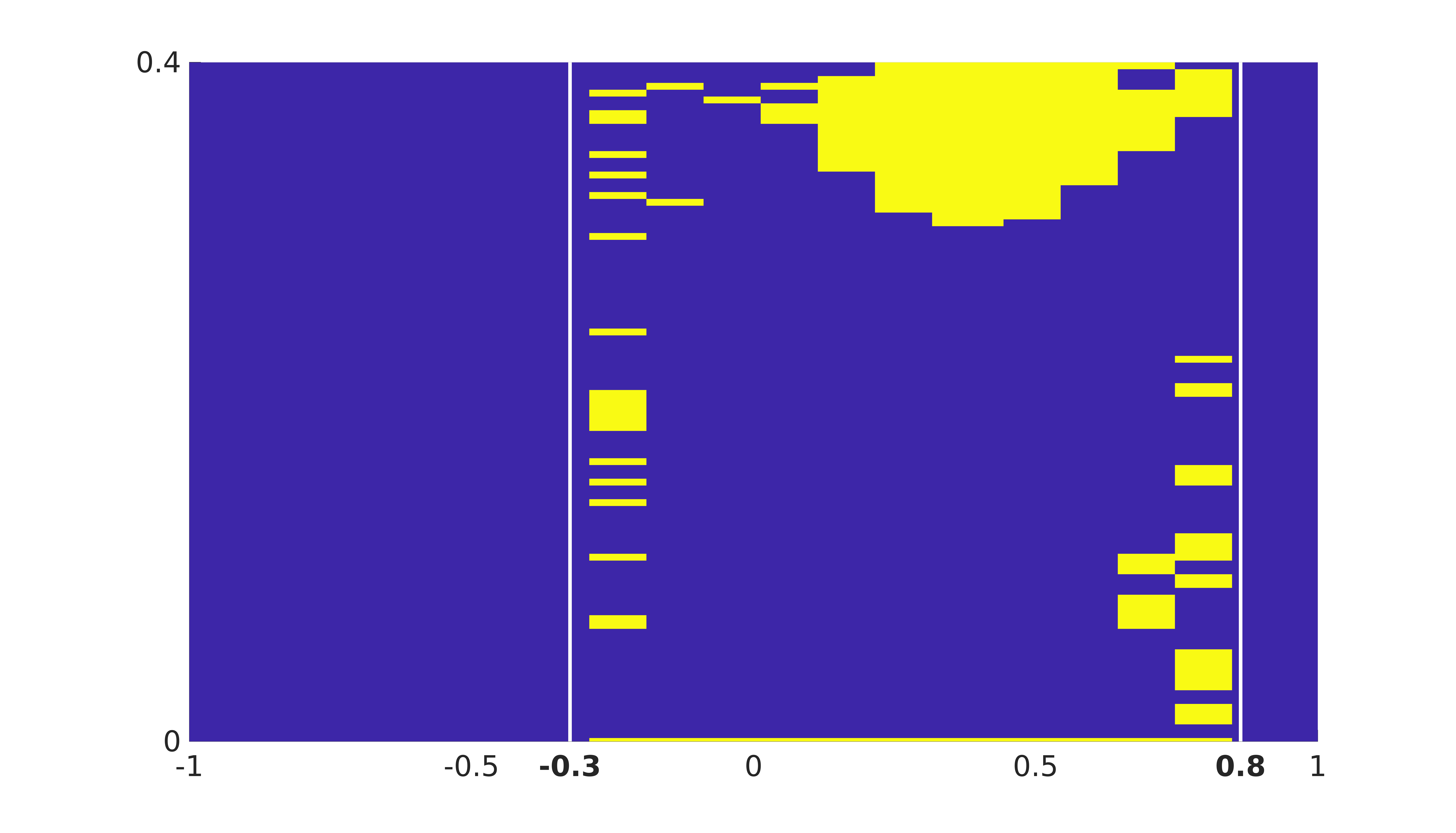}
	\caption{Behavior of the control in time $T=0.4$. The white lines delimit the control region $\omega=(-0.3,0.8)$. The regions in which the control is active are marked in yellow. The atomic nature is lost.}\label{control_down_up_fig}
\end{figure}

Finally, when considering a time horizon $T<T_{\rm min}$ we can notice once more that the solution of \eqref{frac_heat} fails to be controlled to the desired trajectory $\widehat{z}(\cdot,T)$. In fact, Figure \ref{constrained_up_tsmall_fig} shows that the numerical target (displayed in green) computed by employing the tolls of the DyCon computational toolbox for solving the optimization problems \eqref{opt_dycon}-\eqref{opt_dycon_constr} does not totally match the desired target in blue.

\begin{figure}[h]
	\begin{minipage}{0.47\textwidth}
		\centering
		\includegraphics[scale=0.43]{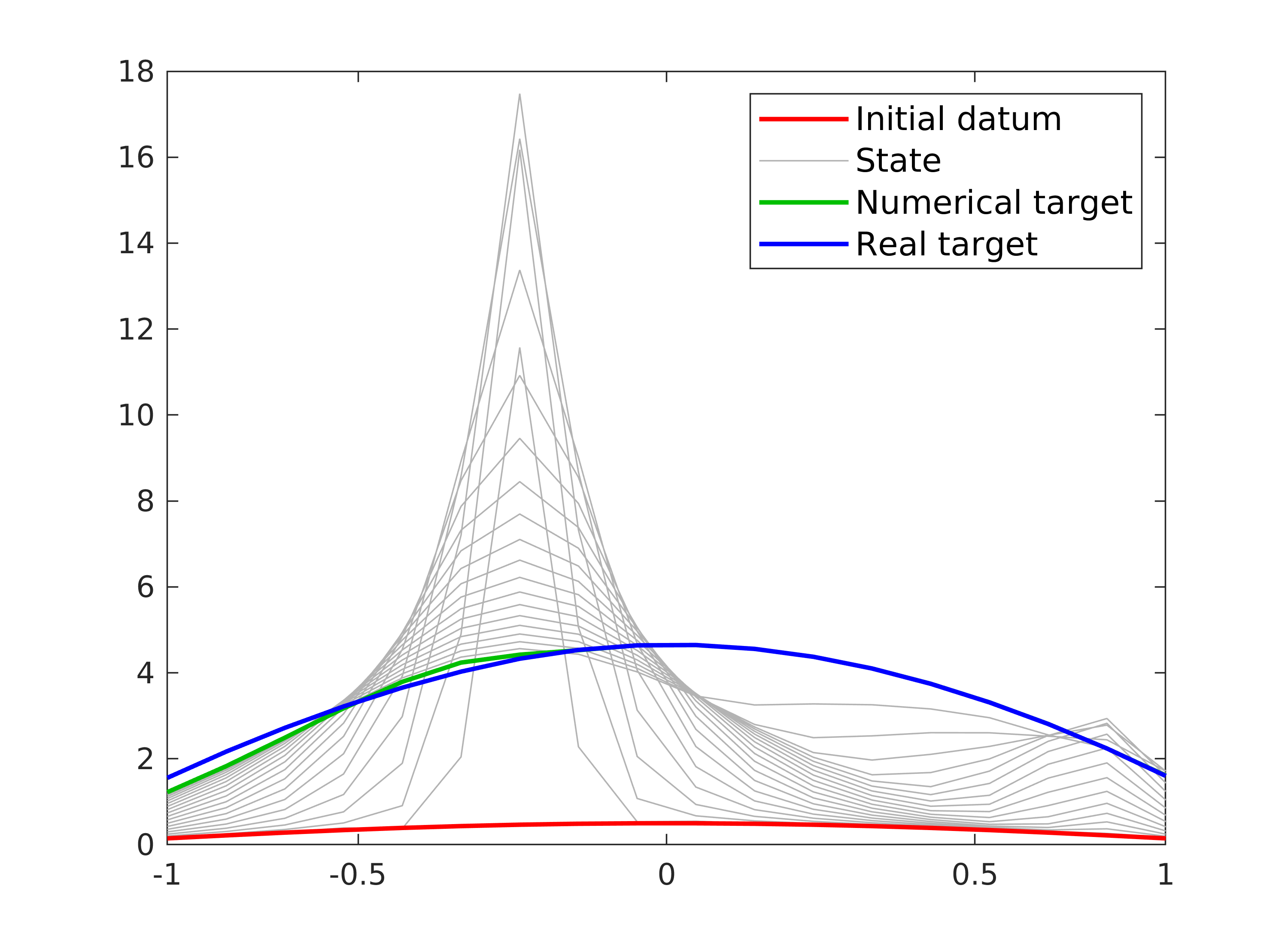}
	\end{minipage}
	\begin{minipage}{0.47\textwidth}
		\centering
		\includegraphics[scale=0.43]{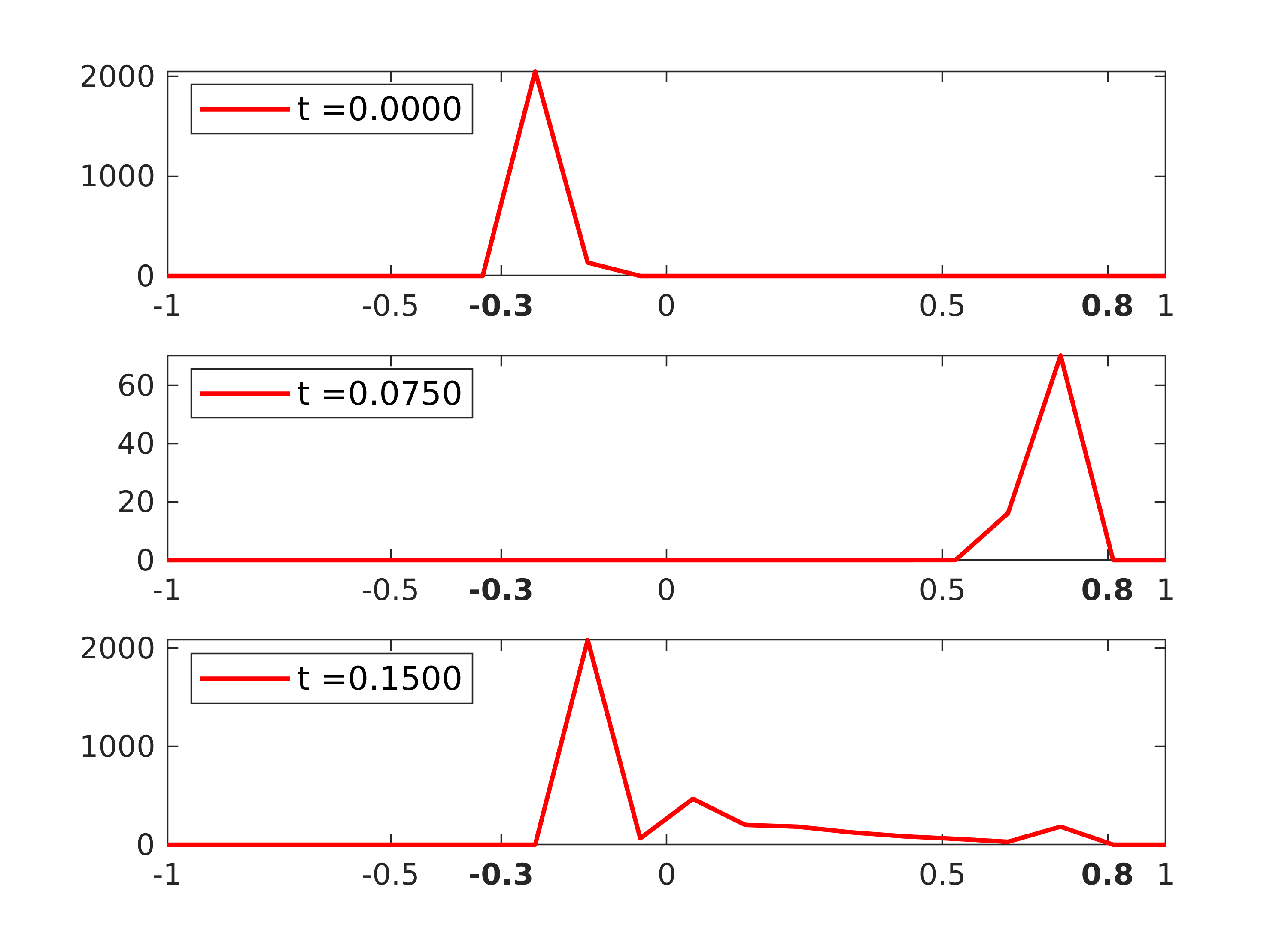}
	\end{minipage}
	\caption{Evolution in the time interval $(0,0.15)$ of the solution of \eqref{frac_heat} with $s=0.8$ (left) and of the control $u$ (right), under the constraint $u\geq 0$. The bold characters highlight the control region $\omega=(-0.3,0.8)$. The equation is not controllable.}\label{constrained_up_tsmall_fig}
\end{figure}

%\newpage 

\section{Concluding remarks}\label{sec-con-rem}

In this paper, we have studied the controllability to trajectories for a one-dimensional fractional heat equation under nonnegativity state and control constraints. 

For $s>1/2$, when the controllability for the unconstrained fractional heat equation holds in any positive time $T>0$ by means of an $L^2$-distributed control, we have shown that the introduction of state or control constraints creates a positive minimal time $T_{\rm min}$ for achieving the same result. Moreover, we have also proved that, in this minimal time, constrained controllability holds with controls in the space of Radon measures. 

Our results, which are in the same spirit of the analogous ones obtained in \cite{loheac2017minimal,pighin2018controllability} for the linear and semi-linear heat equations in one and several space dimensions, are supported by the numerical simulations in Section \ref{numerical_sec}.

We present hereafter a non-exhaustive list of open problems and perspectives related to our work.

\begin{itemize}
	\item[1.] \textbf{Extension to the multi-dimensional case.} Our analysis is based on spectral techniques, and applies only to a one-dimensional fractional heat equation. The extension to multi-dimensional problems on bounded domains $\Omega\subset\RR^N$, $N\geq 1$, requires different tools such as Carleman estimates. Nevertheless, obtaining Carleman estimates for the fractional Laplacian is a very difficult issue which has been addressed only partially, and only for problems defined on the whole Euclidean space $\RR^N$ (see, e.g., \cite{ruland2015unique}). The case of bounded domains remains open and it is quite challenging. As one expects, the main difficulties come from the nonlocal nature of the fractional Laplacian, which makes classical PDEs techniques more delicate or even impossible to use. \\
	
	\item[2.] \textbf{Bang-bang nature of the controls.} In Theorem \ref{control_thm_unconstr} we showed that the fractional heat equation \eqref{frac_heat} is controllable to trajectories by means of $L^\infty$-controls satisfying $\norm{u}{L^\infty(\omega)}=\norm{p}{L^1(\omega)}$, $p$ being the solution of the adjoint equation \eqref{frac_heat_adj}. It is then a natural and interesting question to analyze whether these controls have a bang-bang nature, that is,
	\begin{align*}
		u = \norm{p}{L^\infty(\omega\times(0,T))}\textrm{sign}(p).
	\end{align*}
	To this end, the first step would be to show that the zero set of the solutions of the adjoint equation is of null measure, so that the sign of the adjoint state is well defined. This is true in the case of the classical heat equation, as a consequence of the space-time analyticity properties of the solutions. Nevertheless, we do not know whether the same holds also for the fractional heat equation \eqref{frac_heat} since, in this case, as far as we know no space-time analytic regularity results are available (in fact, it is known that solutions are analytic in time, but the space analyticity is still an open problem). Then, this becomes a very challenging problem, both in PDE analysis and control theory.   \\
	
	\item[3.] \textbf{Constrained controllability from the exterior for the fractional heat equation.} In \cite{warma2018null}, the null-controllability from the exterior for the one-dimensional fractional heat equation has been treated. In particular, it has been proved there that, if $s>1/2$, null controllability is achievable by means of a control function $g$ acting on a subset $\mathcal O\subset (\RR\setminus (-1,1))$. Hence, to address constrained controllability in this framework becomes a very interesting issue.\\
	
	\item[4.] \textbf{Bilinear control for the fractional heat equation.} In many frameworks, for instance in population dynamics, the employment of bilinear (multiplicative) control and positivity constraints arise naturally. Hence, an interesting problem to analyze would be the bilinear  control of the fractional heat equation, which can be formulated in terms of the following controlled equation:
	\begin{equation*}
		\begin{cases}
			z_t+\fl{s}{z} = f(x,t)u, \quad &\mbox{ in }\; (-1,1)\times(0,T),\quad  s\in(0,1),\\
			z=0, &\mbox{ in } (\mathbb R \setminus (-1,1))\times (0,T),\\
			z(\cdot,0)=z_0,\; &\mbox{ in }\; (-1,1).
		\end{cases}
	\end{equation*}
	Often times a bilinear control is built by linearization and fixed point arguments and the linearized problem can be handled, in the one-dimensional case, by Ingham-like inequalities such as \eqref{obs2}. For the local case, this issue has been treated in several references, including \cite{cannarsa2011approximate,cannarsa2017multiplicative,cannarsa2010multiplicative,floridia2014approximate}. The extension of those results to the nonlocal setting will be the subject of a future research work. \\
	
	\item[5.] \textbf{Lower bounds for the minimal constrained controllability time.} In Section \ref{numerical_sec}, we gave some numerical lower bound for the minimal constrained controllability time. Nevertheless, the bounds we presented are not optimal. This raises two very important issues. On the one hand, we shall obtain analytical lower bounds for the controllability time. In particular, to understand how it depends on the order $s$ of the fractional Laplacian is evidently a fundamental issue to be clarified. This question was already addressed in \cite{loheac2017minimal,pighin2018controllability} for the local heat equation but, as we discussed in Section \ref{sec44}, the methodology developed in those works does not apply immediately to our case. Therefore, there is the necessity to adapt the techniques of \cite{loheac2017minimal,pighin2018controllability}, or to develop new ones. On the other hand, we should develop a complete analysis of the efficiency of the numerical method we used for estimating this minimal time, in order to determine the accuracy of our approximation. \\
	
	\item[6.] \textbf{Convergence result for the minimal time.} The minimal time $T_{\rm min}$ in the simulations of Section \ref{numerical_sec} is just an approximation computed by solving numerically the optimization problem \eqref{Topt}-\eqref{Topt_constr}. The validity of these computational result should be confirmed by showing that this minimal time of control for the discrete problem converges towards the continuous one as the mesh-sizes tend to zero. This could be done by adapting the procedure presented in \cite[Section 5.3]{loheac2017minimal}. Nevertheless, we have to mention that, in order to corroborate this procedure, it is required the knowledge of an analytic lower bound for $T_{\rm min}$ which, at the present stage, it is unknown (see point 4 above).
\end{itemize}

\section*{Acknowledgments} The authors would want to thank Debayan Maity (TIFR Centre for Applicable Mathematics, Bangalore, India) for his valuable help in some delicate technical issues of our proofs. A special thanks goes to Deyviss Jesus Oroya (DeustoTech, University of Deusto, Bilbao, Spain) for his contribution to the simulations in Section \ref{numerical_sec}.

\bibliographystyle{acm}
\bibliography{biblio}

\end{document}